\documentclass{amsart}
	
\usepackage{amsmath,amssymb,amsthm,amsfonts,enumerate,url,fancyhdr,mathbbol,mathrsfs,color} 
\usepackage{hyperref}
\usepackage[sort,numbers]{natbib}

\usepackage[utf8]{inputenc}
\def\me{\mathsf{e}}
\def\mv{\mathsf{v}}
\def\:{\thinspace:\thinspace}
\newtheorem{theo}{Theorem}
\newtheorem{lemma}[theo]{Lemma}
\newtheorem{prop}[theo]{Proposition}
\newtheorem{cor}[theo]{Corollary}

\newtheorem{defi}[theo]{Definition}

\newtheorem{assum}[theo]{Assumptions}
\newtheorem{rem}[theo]{Remark}

\newtheorem{exa}[theo]{Example}

\numberwithin{equation}{section}
\numberwithin{theo}{section}
\def\:{\thinspace:\thinspace}

\theoremstyle{definition}

\numberwithin{theo}{section}
\DeclareMathOperator{\Tr}{Tr}

  \def\mG{\mathsf{G}}
  \def\mV{\mathsf{V}}
  \def\mE{\mathsf{E}}
  \def\Efun{\mathscr{E}}
  \def\Lfun{\mathscr{L}}
  \def\Ffun{\mathscr{F}}

  \def\mT{\mathsf{T}}
\def\mv{\mathsf{v}}
 \def\me{\mathsf{e}}
 \def\mw{\mathsf{w}}
  \def\mW{\mathsf{W}}
  \def\mf{\mathsf{f}}

\title{Parabolic theory of the discrete $p$-Laplace operator}
 \author{Delio Mugnolo}

\address{Delio Mugnolo, Institut f\"ur Analysis, Universit\"at Ulm, 89069 Ulm, Germany}
\email{delio.mugnolo@uni-ulm.de}
\thanks{I would like to thank Daniel Lenz, Robin Nittka, and Ren\'e Pr\"opper for fruitful discussions. I am also grateful to the anonymous referees for their attentive reading and essential suggestions which led to substantial improvement of this article. This research has been supported by the Land Baden--W\"urttemberg in the framework of the \emph{Juniorprofessorenprogramm} -- research project on ``Symmetry methods in quantum graphs''.}

\begin{document}
\begin{abstract}We study the discrete version of the $p$-Laplace operator. Based on its variational properties we discuss some features of the associated parabolic problem. We prove well-posedness of the problem and obtain information about positivity and comparison principles as well as compatibility with the symmetries of the underlying graph. Our methods consist in an interplay of the theory of subdifferentials and of combinatorial methods. We conclude briefly discussing the variational properties of a handful of nonlinear generalized Laplacians appearing in different parabolic equations.
\end{abstract}
\maketitle

\section{Introduction}

The \emph{discrete Laplacian} is a well-known object in graph theory. In his seminal investigation on electric circuits~\cite{Kir47}, Kirchhoff introduced it as
\[
\Delta:=D-A,
\]
where $D$ is the (diagonal) degree matrix and $\mathcal A$ the adjacency matrix of a {finite} undirected graph $\mG$. He then went on to observe that for any orientation of the graph the associated (signed) incidence matrix $\mathcal I$ satisfies 
\begin{equation}\label{kirchlapl}
\Delta=\mathcal I \mathcal I^T,
\end{equation}
that this representation permits to prove that all the minors of $\Delta$ have the same value, and that this value is a natural number that in fact agrees with the number of spanning trees of $\mG$; moreover, $0$ is always an eigenvalue whose multiplicity coincides with the number of connected components of $\mG$. 

A possible reason for considering $\Delta$ as a discrete analog of the usual Laplace operator is that, if one sets up a system of linear equations for the potential $f$ at each node $\mv$ of an electric circuit in accordance with the (linear) laws of Kirchhoff and Ohm, then 
\[
\Delta f=0,
\]
i.e., $f$ satisfies a Laplace-type equation like its pendant {in the continuum -- the electrostatic potential -- does}: this is well explained in~\cite[Chapt.~1]{DoySne84}. (However, the discrete Laplacian is by~\eqref{kirchlapl} positive definite whereas the common Laplace operator in the continuum setting is negative definite). Also after Kirchhoff's investigation, $\Delta$ has proved a remarkable object in algebraic graph theory. The r\^ole of the second smallest eigenvalue of $\Delta$ in relation with the study of connectivity properties of a graph has been emphasized since the 1970s by Fiedler and others, cf.~the survey article~\cite{Moh91}. Moreover, it is nowadays known (see e.g.~\cite[Chapt.~3]{Mug13}) that the discrete Laplacian (possibly after a suitable re-normalization) is tightly related to the Dirichlet-to-Neumann (i.e., the voltage-to-current) operator acting on functions defined on the metric structure associated with the discrete graph -- the so-called \emph{metric graph}.

 A thorough theory of linear discrete elliptic operators is nowadays available. This field has experienced successful interactions between the communities of researchers working on graph theory, functional analysis, stochastics and potential theory, see e.g.~\cite{Moh91,KelLen09,JorPea09} or the monograph~\cite{Woe00}. 

\bigskip
In the 1960s and 1970s, pioneering investigations by a group of researchers including Minty, Rockafellar, and Zemanian aroused broad interest in the theory of nonlinear electric circuits. 
Remarkably, these authors were also among those by whom, in the same years,  the theory of monotone operators and subdifferentials was being substantially developed. It is therefore no surprise that many investigations focused on those nonlinear electric circuits whose associated voltage-to-current operator is maximal monotone. In the same years, the $p$-Laplace operator $\Delta_p$ {in the continuum}, i.e., the subdifferential of the energy functional
\begin{equation}
\label{plaplfunctcont}
L^p(\Omega)\ni u\mapsto \frac{1}{p}\|\nabla u\|^p_{L^p}\in [0,\infty]
\end{equation}
began to be studied by Aronsson, DiBenedetto, Ural'ceva, and many others (cf.~\cite{Dra07} for an introduction and survey on this topic). It was in this cultural climate that Nakamura and Yamasaki introduced in~\cite{NakYam76} on an {infinite} graph $\mG$ with node set $\mV$ the mapping
\footnote{ Throughout this paper we write $\mv \sim \mw$ if there is an edge between $\mv$ and $\mw$.}
\[
\Efun_p:\mathbb R^\mV\ni f\mapsto \frac{1}{p}\sum_{\substack{\mv,\mw\in \mV \\ \mv\sim \mw}} |f(\mv)-f(\mw)|^p\in [0,\infty].
\]
While Yamasaki and coauthors kept on studying properties of nonlinear electric networks with respect to this functional for several years, this object was seemingly overlooked by most other researchers until the early 1990s, when $\Efun_p$ reappeared as part of the final remark of a short note~\cite{Soa93} on linear potential theory. Ever since, many authors have discussed properties of this energy functional and of its subdifferential, the \emph{discrete $p$-Laplacian} defined (at least for finitely supported functions) by
\[
\Delta_{p} f(\mv):=\sum_{\substack{\mw\in \mV \\ \mw\sim \mv}}|f(\mv)-f(\mw)|^{p-2}(f(\mv)-f(\mw)),\qquad \mv\in \mV.
\]
(Observe that for $p=2$ we recover Kirchhoff's discrete Laplacian. {Also observe that while one could in principle impose a labeling yielding an order on $\mV$, and accordingly a lexicographical order on $\mE$, this seems to be somehow unnatural and unsatisfactory. Accordingly, we will rather look for assumptions implying unconditional  convergence of the series considered in this paper}). 

 Most investigations have been devoted to potential theory and eigenvalue problems: we mention, among others,~\cite{Sal97,HolSoa97b,Amg06} and a long series of articles by Agarwal, O'Regan and coauthors that begins with~\cite{JiaChuReg04}. Several works have also emphasized the r\^ole of the discrete $p$-Laplacian, in particular for the limiting case of $p=1$, in spectral clustering and image processing: see among others~\cite{ZhoSch05,ElmLezBou08,BuhHei09,SzlBre10}. In particular, in~\cite[\S 5.2]{GraPol10} also the associated parabolic problems are briefly referred to, and in~\cite[\S 3.2.2.1]{GraPol10} also the connections with nonlinear circuits are outlined. {Two recent fascinating developments initiated by the late Oded Schramm and his coauthors in~\cite{BenSch09,PerSchShe09} link the discrete $p$-Laplacian with the sphere packing problem and with the tug-of-war-theory.}

In view of the then ongoing investigations on the $p$-Laplacian in the continuum, Nakamura--Yamasaki's introduction of $\Efun_p$ was motivated by the fact that the finite difference 
$$\mathcal I^T f(\mv,\mw):=f(\mv)-f(\mw)$$ 
can be seen as a discrete analog of (minus) a directional derivative -- this fact is classically used in numerical analysis, giving rise to the notion of \emph{backward difference operator}. If we impose an orientation on a graph $\mG$ and regard functions from $\mV$ to $\mathbb R$ as \emph{difference} 0-forms (or \emph{0-chains}) and functions from $\mE$ to $\mathbb R$ as \emph{difference} 1-forms (or \emph{1-chains}), the corresponding derivation is exactly ${\mathcal I}^T:\mathbb R^\mV \to \mathbb R^\mE$, the transpose of the oriented $|\mV|\times |\mE|$ incidence matrix
\begin{equation}\label{incidintro}
{\mathcal I}:={\mathcal I}^+-{\mathcal I}^-, 
\end{equation}
where 
${\mathcal I}^+:=({\iota}^+_{\mv \me})$ and ${\mathcal I}^-:=({\iota}^-_{\mv \me})$ are defined by
$${\iota}^+_{\mv \me}:=\left\{
\begin{array}{ll}
1 & \hbox{if } \mv \hbox{ is initial endpoint of } \me,\\
0 & \hbox{otherwise},
\end{array}\right.\qquad 
{\iota}^-_{\mv \me}:=\left\{
\begin{array}{ll}
1 & \hbox{if } \mv \hbox{ is terminal endpoint of } \me,\\
0 & \hbox{otherwise},
\end{array}\right.$$
respectively. 
Hence, Yamasaki's functional $\Efun_p$ can be written as
\[
\Efun_p:= f\mapsto \frac{1}{p}\|{\mathcal I}^T f\|^p_{\ell^p(\mE)},
\]
and
\begin{equation}\label{plaplfirstdef}
\Delta_p: f\mapsto\mathcal I(|\mathcal I^T f|^{p-2}\mathcal I^T f),
\end{equation}
in apparent analogy with the $p$-Laplacian {in the continuum} and its associated energy functional.

Just like in the continuous case, where (after applying the divergence theorem) the conservation law
\[
\frac{d}{dt}\int_\Omega \varphi(t,x)dx =-\int_{\Omega}{\rm div} j(t,x)dx+\int_\Omega f (t,x)dx\\
\]
 is considered, in the discrete case we have the conservation law 
\[
\frac{d}{dt}\sum_{\mv \in \mV} \varphi(t,\mv)=-\sum_{\mv\in \mV}(\mathcal I j)(t,\mv)+\sum_{\mv\in \mV}f (t,\mv),
\]
{where (depending on the model) $\varphi$ denotes a temperature or perhaps a density of a chemical substance or a population,  $j$ is the flux function, and $f$ encodes the production or destruction of $\varphi$ due to reactions or other phenomena not otherwise modeled by the equation.}
Then, Fick's law for $j$
\[
j(t,\me)=c\left(\mathcal I^T \varphi(t,\me)\right),
\]
for a general function $c$, can be used to derive a (time-continuous, space-discrete) differential equation governing a flow on $\mG$, cf.~\cite[\S~2.5.5]{GraPol10}. Indeed, for $c(x)=x$ we obtain exactly the linear {(inhomogeneous)} heat equation for the discrete Laplacian. Considering different terms $c$ we obtain different equations. In particular, opting for the nonlinear Darcy's law
\[
c(x):=|x|^{p-2}x,
\]
we end up with 
\begin{equation}
\label{pHE$^N_p$}
\frac{d}{dt}\sum_{\mv \in \mV} \varphi(t,\mv)=-\sum_{\mv\in \mV}\mathcal I(|\mathcal I^T \varphi|^{p-2} \mathcal I^T \varphi)(t,\mv)+\sum_{\mv\in \mV}f (t,\mv).
\end{equation}
In view of~\eqref{plaplfirstdef}, this is the \emph{summed} version of the parabolic problem associated with the discrete $p$-Laplacian $\Delta_p$ -- let us call it the \emph{discrete $p$-heat equation}. Observe that, unlike its counterpart {in the continuum}, it can be seen as a \emph{backward} evolution equation.

 The present paper is devoted to study the elementary features of this equation and is structured as follows. After an introductory section where the relevant functional setting is introduced, we proceed by formulating in Section~\ref{sec:3} our main well-posedness result for the discrete $p$-heat equation. 
 {As already observed in the linear case in~\cite{HaeKelLen12}, it turns out that in the case of infinite graph two natural $p$-Laplacians -- a Dirichlet-type and a Neumann-type $p$-Laplacian -- can be associated with the discrete $p$-heat equation: Well-posedness holds for both, but in the Dirichlet case a Galerkin-type method can be applied to yield a more constructive proof.} Indeed, unlike in the usual context of PDEs on domains  we can even interpret the converging Galerkin sequence as a sequence of solutions to the same equation on induced subgraphs (i.e., on growing subsystems). Using the theory of nonlinear Dirichlet forms, we can show that these solutions are associated with semigroups of nonlinear contraction on all $\ell^q(\mV)$-spaces, $1\le q\le \infty$, {as well as on the space $c_0(\mV)$ of sequences vanishing at infinity.} 
 This is done by discussing the invariance property of several relevant closed convex subsets under the $C_0$-semigroup generated by $-\Efun_p$.  This is more delicate than in the continuous case, since the discrete $p$-Laplacian is not a local operator. We show that the nonlinear $C_0$-semigroup generated by the discrete $p$-Laplacian consists of irreducible, sub-Markovian operators. In the special case of $p=2$, these results had already been obtained in~\cite{KelLen09}. {It should be recalled that a different operator theoretical approach to this problems has been developed in~\cite{JorPea09}, where the discrete Laplacian has been instead studied in the Hilbert space $\{f\in \mathbb R^\mV:\Efun_2(f)<\infty\}/{\mathbb R}$ -- in accordance with the setting used in all Yamasaki's papers, beginning with~\cite{Yam77}.}


Energy methods are applied in Section~\ref{sec:symm} also to discuss how these $C_0$-semigroups interact with the symmetries of the graph, {based on the notion of almost equitable partition of a graph. Such a notion, which is quite classical in algebraic graph theory, proves most suitable to investigate properties of~\eqref{pHE$^N_p$}. Investigation of symmetries of~\eqref{pHE$^N_p$} is in fact not a novelty, at least in the linear case, and we conclude Section~\ref{sec:symm} comparing the notion of almost equitable partitions with other symmetry concepts.}

Finally, in {Section~\ref{generalized}} we overview a few popular generalizations of the discrete Laplacian, propose some more and suggest how the variational structure of the $p$-heat equation can be generalized to discuss a broad class of discrete nonlinear diffusion-type problems, including a discretized porous medium equation.

Because the discrete $p$-heat equation is simply a dynamical system, most of our results follow from the standard theory of ordinary differential equations whenever $\mG$ is a finite graph. The actually interesting case is hence that of infinite graphs. We refer to~\cite[Chapter~8]{Die05} or~\cite[Chapter~1]{Woe00} for an introduction to infinite graph theory. {On the analytical side, we have tried to keep this article as self-contained as possible by recalling in the Appendix all the abstract results  we are going to need -- but for more details and precise proofs we refer to~\cite{Bre73,Sho97}, two standard references for nonlinear Cauchy problems on Hilbert spaces. An exposition of some basic properties of discrete vector analysis, including the discrete versions of Leibniz' rule and Stokes' theorem, have been outlined e.g. in~\cite{DodKar88}, cf.\ also~\cite[\S~2]{RigSalVig97} for weighted versions and~\cite{GraPol10} for a more elementary introduction to this topic.
}
 
\section{General setting and the energy functional}\label{sec:2}

We consider throughout a (finite or countable) set $\mV$ and a (non-symmetric) relation 
\[
\mE\subset\{(\mv,\mw)\in \mV\times \mV:\mv\not= \mw\}
\]
such that for any two elements $\mv,\mw\in \mV$ at most one of the pairs $(\mv,\mw),(\mw,\mv)$ belongs to $\mE$. We refer to the elements of $\mV$ and $\mE$ as \emph{nodes} and \emph{edges}, respectively, and regard $\mG$ as a graph whose incidence matrix is $\mathcal I$ introduced in~\eqref{incidintro}. If $\me=(\mv,\mw)$, we denote the initial and terminal endpoint of $\me$ by 
\[
\me_+:=\mv\quad\hbox{ and }\quad\me_-:=\mw,\qquad \hbox{ respectively, where }\me=(\mv,\mw).
\] 

Inspired by~\cite[\S~1]{KelLen09}, throughout this article we consider $\mG$ as a \emph{weighted graph}, i.e., as a quadruple $(\mV,\mE,\mu,\nu)$ where $\mu:\mE\to [0,\infty)$ and $\nu:\mV\to [0,\infty)$ (the unweighted case corresponds to $\mu\equiv 1$ and $\nu\equiv 1$). 

\begin{rem}
Actually, by construction $\mG$ is oriented. This is reflected already in the definition of the (signed) incidence matrix $\mathcal I$. However, orientation is only a technical tool that will make some parametrizations easier: Unless otherwise underlined, throughout most of this paper only the absolute value of the numbers
\[
\mathcal I^T f(e)=\sum_{\mv\in \mV} \iota_{\mv \me}f(\mv)=f(\me_+)-f(\me_-)
\]
will play a r\^ole. Thus, all our results (with the exception of those in Section~\ref{sec:symm}) are still valid if any of the edges is given the opposite orientation.
\end{rem}

We have already introduced the (possibly infinite) incidence matrix 
\[
\mathcal I:\mathbb R^\mE \to \mathbb R^\mV\quad \hbox{ and its transposed }\quad \mathcal I^T:\mathbb R^\mV \to \mathbb R^\mE.
\]
Along with them, the structure of the \emph{undirected} graph associated with $\mG$ is mirrored by the $|\mV|\times |\mV|$ \emph{adjacency matrix} $\mathcal A:=(\alpha_{\mv\mw})$ defined by
\begin{equation}\label{adjdefi}
{\alpha}_{\mv \mw}:=\left\{
\begin{array}{ll}
\mu(\me) & \hbox{if } \me=(\mv,\mw) \hbox{ or } \me=(\mw,\mv)\in \mE,\\
0 & \hbox{otherwise.}
\end{array}\right.
\end{equation}

{The following definition arises from a natural interplay of the usual definitions of \emph{degree} in weighted, finite graph theory, and of \emph{local finiteness} in unweighted, infinite graph theory, cf.\ e.g.~\cite[\S~1.5]{Chu97} and~\cite[\S~8.1]{Die05}), respectively. It is similar to but different from the notion of almost local finiteness in~\cite{SoaYam93}.}

\begin{defi}\label{defi:basic}
A graph $\mG:=(\mV,\mE,\mu,\nu)$ is called \emph{outward locally finite} if its \emph{outdegree} function satisfies
\[
{\rm deg}^+(\mv):=\sum_{\me \in\mE}\iota^+_{\mv\me}\mu(\me)\le M^+_\mv\qquad \hbox{ for all }\mv\in\mV\hbox{ and some }M^+_\mv>0.
\]
It is called \emph{inward locally finite} if its \emph{indegree} function satisfies
\[
{\rm deg}^-(\mv):=\sum_{\me \in\mE}\iota^-_{\mv\me}\mu(\me)\le M^-_\mv\qquad \hbox{ for a	ll }\mv\in\mV\hbox{ and some }M^-_\mv>0.
\]
It is called \emph{locally finite} if it is both inward and outward locally finite, i.e., if its \emph{degree} function satisfies
\[
{\rm deg}(\mv):={\rm deg}^+(\mv)+{\rm deg}^-(\mv)\le M_\mv\qquad \hbox{ for all }\mv\in\mV\hbox{ and some }M_\mv>0.
\]
If ${\rm deg}^+\in O(\nu)$, i.e., if there exists $M^+>0$ s.t.
\begin{equation*}
\label{nudeg}
{\rm deg}^+ (\mv)\le M^+ \nu(\mv)\qquad \hbox{for all }\mv \in \mV,
\end{equation*}
then $\mG$ is called  \emph{outward uniformly locally finite}. We define \emph{inward uniform local finiteness} and \emph{uniform local finiteness} likewise.

We say that $\mG$ has \emph{finite surface} and \emph{finite volume} if
\[
|\mV|_\nu:=\|\nu\|_{\ell^1(\mV)}<\infty\qquad \hbox{and}\qquad |\mE|_\mu:=\|\mu\|_{\ell^1(\mE)}<\infty,
\]
respectively.
\end{defi}

{
\begin{rem}
\label{fvfsulf}
The weighted version of the handshake lemma states that
\[
|\mE|_\mu =\frac12 \sum_{\mv \in \mV}{\rm deg}(\mv).
\]
Hence, any graph with finite volume is necessarily uniformly locally finite.
\end{rem}
}


\begin{assum}\label{standing}
Throughout this paper we impose the following assumptions on $\mG:=(\mV,\mE,\mu,\nu)$.
\begin{itemize}
\item $\mu(\me)>0$ for all $\me\in \mE$.
\item $\nu(\mv)>0$ for all $\mv\in \mV$.
\item $\mG$ is locally finite.
\end{itemize}
\end{assum}

The weights $\mu,\nu$ define in a natural way a metric and measure structure on $\mG$. {As we will see in Lemma~\ref{identlp}, there is a natural Laplace--Beltrami-type discrete operator associated with this metric, but we are going to discuss more general (non-degenerate) elliptic operators as well, possibly in non-divergence form.} {To this aim, we want to allow for a large class of weights $a$ and $d$, in pretty much the same way one may want to consider general elliptic operators on a Riemannian manifold endowed with a given metric.} In order to avoid degeneracies, certain compatibility conditions should hold: {Roughly speaking, we want $a$ to behave asymptotically like $\mu$, and $d$ to behave asymptotically as $\nu$.}
 
\begin{assum}\label{aellipt}
Throughout this paper the functions $a:\mE\to \mathbb R$ and $d:\mV\to \mathbb R$ are assumed to satisfy the following conditions.
\begin{itemize}
\item There exist $\kappa,K>0$ s.t.
\[
\kappa a(\me)\le \mu (\me)\le K a(\me)\qquad \hbox{for all }\me \in \mE.
\]
\item There exist $\theta,\Theta>0$ s.t.
\[
\theta d(\mv)\le \nu (\mv)\le \Theta d(\mv)\qquad \hbox{for all }\mv \in \mV.
\]
\end{itemize}
\end{assum} 

Clearly, if there exist $\zeta,Z>0$ s.t.
\[
\zeta \le \mu\le Z\qquad \hbox{ (resp., s.t. }\zeta \le \nu\le Z\hbox{),}
\]
then one is free to choose $a$ (resp., $d$) to be constant.

For $p\in [1,\infty]$ we will consider the weighted sequence spaces $\ell^p_a(\mE)$ and $\ell^q_d(V)$ defined by
\[
\|u\|_{\ell^p_a}^p := \sum_{\me \in \mE} |u(\me)|^p a(\me)\qquad \hbox{and}\qquad \|f\|_{\ell^q_d(\mV)}^q:=\sum_{\mv \in \mV} |f(\mv)|^q d(\mv),\qquad p\in [1,\infty),
\]
{
or
\[
\|u\|_{\ell^\infty_a}:= \sup_{\me \in \mE} |u(\me)| a(\me)\qquad \hbox{and}\qquad \|f\|_{\ell^\infty_d(\mV)}:=\sup_{\mv \in \mV} |f(\mv)| d(\mv),
\]
}
along with the vector space
$$w^{1,p,2}_{a,d}(\mV):=\left\{f\in \ell^2_d(\mV):\mathcal I^T f\in \ell^p_a(\mE)\right\}.$$ 
In other words,
$$w^{1,p,2}_{a,d}(\mV)=\left\{f\in \ell^2_d(\mV):\|\mathcal I^T f\|^p_{\ell^p_a}=\sum_{e\in \mE} a(\me) |f(\me_+)-f(\me_-)|^p<+\infty\right\}.$$

\begin{rem}\label{equiv-norms}
(1) Let $p\in [1,\infty]$. Observe that another way of expressing uniform local finiteness of $\mG$, i.e., that ${\rm deg}\in O(\nu)$, is saying that $\ell^p_{\rm deg}(\mV)$ is continuously embedded in $\ell^p_\nu(\mV)$.

(2) Also observe that owing to the Assumptions~\ref{aellipt} the weighted norms $\|\cdot\|_{\ell^p_\mu}$ and $\|\cdot\|_{\ell^p_a}$, and also $\|\cdot\|_{\ell^p_\nu}$ and $\|\cdot\|_{\ell^p_d}$ are equivalent. In particular, $w^{1,p,2}_{\mu,\nu}$ and $w^{1,p,2}_{a,d}$ agree.

{(3) Clearly, $\mG$ has finite surface if and only if the function ${\bf 1}$ of constant value 1 is in $\ell^1_\nu(\mV)$ if and only if ${\bf 1}$ is in $\ell^q_d(\mV)$ for all $q\in [1,\infty)$ (we are using the Assumptions~\ref{aellipt} in the latter equivalence). Hence, it follows from Hölder's inequality that if $\mG$ has finite surface, then $\ell^p_d(\mV)=\ell^q_d(\mV)$ for all $p,q\in [1,\infty)$, and likewise if $\mG$ has finite volume, then $\ell^p_d(\mE)=\ell^q_d(\mE)$ for all $p,q\in [1,\infty)$.}
\end{rem}

In the theory of discrete calculus, the rule of thumb is to perform the following substitutions:
\begin{itemize}
\item scalar functions $\to$ vectors of the node space,
\item vector fields $\to$ vectors of the edge space,
\item gradient of a scalar function $\varphi$ at a point $x$ $\to$ evaluation of $\mathcal I^T \varphi$ at an edge $\me$,
\item divergence of a vector field $\psi$ at a point $x$ $\to$ evaluation of $\mathcal I \psi$ at a node $\mv$.
\end{itemize}

As already anticipated by the chosen notation, $w^{1,p,2}_{a,d}(\mV)$ will play the r\^ole of a weighted Sobolev space. 

{
\begin{lemma}\label{propw12p}
For all $p\in [1,\infty]$, $w^{1,p,2}_{a,d}(\mV)$ is a Banach space with respect to the norm defined by
\[
\|f\|_{w^{1,p,2}_{a,d}}:=\|f\|_{\ell^2_d}+\|\mathcal I^T f\|_{\ell^p_a}.
\]
For all $p\in [1,\infty]$, $w^{1,p,2}_{a,d}(\mV)$ is continuously and densely embedded into $\ell^2_d(\mV)$.
If moreover $p\in [1,\infty)$, then $w^{1,p,2}_{a,d}(\mV)$ is separable in $\ell^2_d(\mV)$. If $p\in (1,\infty)$, then $w^{1,p,2}_{a,d}(\mV)$  is uniformly convex and hence reflexive.
\end{lemma}
\begin{proof}
Completeness of $w^{1,p,2}_{a,d}(\mV)$ holds because for all $p,q\in [1,\infty)$, $\ell^q_a(\mV)$ and $\ell^p_d(\mE)$ are Banach spaces. Also continuity of the embedding is clear, by definition of the norm of $w^{1,p,2}_{a,d}(\mV)$.
Considering
$$T:w^{1,p,2}_{a,d}(\mV)\ni f\mapsto (f,\mathcal I^T f)\in \ell^2_d(\mV)\times \ell^p_a(\mE),$$
which is an isometry if the Cartesian product on the right is endowed with the $\ell^1$-norm, shows that $w^{1,p,2}_{a,d}$ is uniformly convex for all $p\in (1,\infty)$, since  so are $\ell^p_d(\mV)$ and $\ell^q_a(\mE)$. 
Finally, since the space $c_{00}(\mV)$ of functions with finite support is dense in $\ell^2_d(\mV)$, so is $w^{1,p,2}_{a,d}(\mE)$.
\end{proof}
Even if $w^{1,p,2}_{a,d}(\mV)$ is separable (for $p\not=\infty$), finding a total sequence in it is not obvious. Indeed, it is known that in general $c_{00}(\mV)$ is not dense in the space $\{f\in \mathbb R^\mV:\Efun_2(f)<\infty\}/{\mathbb R}$ and hence that the canonical basis $(\delta_\mv)_{\mv\in \mV}$ does not yield a total sequence of $w^{1,p,2}_{a,d}(\mV)$, cf.\ e.g.~\cite[Thm.~2.12]{Woe00}.}

{
Inspired by the classical theory of Sobolev spaces, where  the space $\mathring{W}^{1,p}(\Omega)$ is defined as the closure of the test functions with compact support in the norm of $W^{1,p}(\Omega)$, we can consider the closure of the space $c_{00}(\mV)$ of finitely supported functions from $\mV$ to $\mathbb R$ in the norm of $w^{1,p,2}_{a,d}(\mV)$: Let us denote it by $\mathring{w}^{1,p,2}_{a,d}(\mV)$. Since it is a closed subspace of $w^{1,p,2}_{a,d}(\mV)$, the following is immediate in view of Lemma~\ref{propw12p}.
\begin{cor}\label{propw12pD}
For all $p\in [1,\infty]$, $\mathring{w}^{1,p,2}_{a,d}(\mV)$ is a Banach space with respect to the norm of $w^{1,p,2}_{a,d}(\mV)$. For all $p\in [1,\infty]$, $\mathring{w}^{1,p,2}_{a,d}(\mV)$ is continuously and densely embedded into $\ell^2_d(\mV)$.
If moreover $p\in [1,\infty)$, then $\mathring{w}^{1,p,2}_{a,d}(\mV)$ is separable in $\ell^2_d(\mV)$. If $p\in (1,\infty)$, then $\mathring{w}^{1,p,2}_{a,d}(\mV)$  is uniformly convex and hence reflexive.
\end{cor}
}

The following result is a generalization of~\cite[Prop.~6]{Car11} to weighted graphs.

\begin{lemma}\label{lemma:car11}
Under our standing Assumptions~\ref{standing}, let $p\in [1,\infty]$. If $\mG$ is outward uniformly (resp., inward uniformly, uniformly) locally finite, then $\mathcal I^+$ (resp., $\mathcal I^-,\mathcal I$) is bounded from $\ell^p_a(\mE)$ to $\ell^p_d(\mV)$. The converse implication holds for $p\in [1,\infty)$. It also holds for $p=\infty$ if additionally there exist $\tilde{\mu},\tilde{\nu}>0$ s.t.\ $\mu(\me)\le \tilde{\mu}$ and $\tilde{\nu}<\nu(\mv)$ for all $\me\in \mE$ and $\mv\in \mV$.

In particular, $\mG$ is uniformly locally finite if and only if $w^{1,p,2}_{a,d}(\mV)=\ell^2_d(\mV)$ for all $p\in [2,\infty)$.
\end{lemma}

\begin{proof}
Take $f:\mV \to \mathbb R$ and observe that for all $\me\in \mE$  there exists exactly one $\mv\in \mV$ s.t.\ $\iota_{\mv\me}^+\not=0$, hence
\[
\sum_{\mv \in \mV} \iota^+_{\mv\me} |f(\mv)|^p=|f(\me_+)|^p=\sup_{\mv\in \mV} \iota^+_{\mv\me}|f(\mv)|^p\qquad \hbox{for all }\me\in \mE \hbox{ and all }p\in [1,\infty),
\]
 
Take first $p\in [1,\infty)$. By Fubini's theorem, one has for all $f:\mV \to \mathbb R$
\begin{eqnarray*}
\|\mathcal I^{+T}f\|^p_{\ell^p_\mu} &=& \sum_{\me \in \mE} |f(\me_+)|^p \mu(\me)\\
&=& \sum_{\me \in \mE}\left( \sum_{\mv \in \mV} |f(\mv)|^p \iota^+_{\mv\me}\right)\mu(\me)\\
&=& \sum_{\mv \in \mV} |f(\mv)|^p \sum_{\me \in \mE} \iota^+_{\mv\me}\mu(\me)\\
&=& \sum_{\mv \in \mV} |f(\mv)|^p{\rm deg}^+(\mv)=\|f\|^p_{\ell^p_{\rm deg}}.
\end{eqnarray*}
This shows that $\mathcal I^{+T}$ is an isometry from $\ell^p_\mu(\mE)$ to $\ell^p_{\deg}(\mV)$, hence by Remark~\ref{equiv-norms}.(1) it is bounded from $\ell^p_\mu(\mE)$ to $\ell^p_{\nu}(\mV)$ if and only if $\mG$ is uniformly locally finite. One concludes that $\mathcal I^{+T}$ is bounded from $\ell^p_a(\mE)$ to $\ell^p_d(\mV)$ if and only if $\mG$ is uniformly locally finite.

For $p=\infty$ one has
\begin{eqnarray*}
\|\mathcal I^{+T}f\|_{\ell^\infty_\mu} &=& \sup_{\me \in \mE} |f(\me_+)| \mu(\me)\\
&=& \sup_{\me \in \mE}\sup_{\mv\in \mV} \iota^+_{\mv\me}|f(\mv)| \mu(\me)\\
&\le & \sum_{\me \in \mE}\sup_{\mv\in \mV} \iota^+_{\mv\me}|f(\mv)| \mu(\me)\\
&= & \sup_{\mv\in \mV}|f(\mv)| \sum_{\me \in \mE}\iota^+_{\mv\me} \mu(\me)\\
&= & \sup_{\mv\in \mV}|f(\mv)| {\rm deg}(\mv)=\|f\|_{\ell^\infty_{\rm deg}}.
\end{eqnarray*}
Again, if $\mG$ is uniformly locally finite this inequality suffices to say that $\mathcal I^{+T}$ is bounded from $\ell^\infty_\mu(\mE)$ to $\ell^\infty_{\deg}(\mV)$.

To prove the converse implication, take a sequence $(\mv_n)_{n\in \mathbb N}\subset \mV$ s.t.\ 
\[
n\nu(\mv_n)\le {\rm deg}^+(\mv_n)
\]
and consider the functions $u_n:\mE\to \mathbb R$ defined for all $n\in \mathbb N$ by
\[
u_n(\me):={\mathbf 1}_{\mE^+_n},
\]
where $\mE^+_n:=\{\me\in \mE:\iota^+_{\mv_n \me}\not=0\}$, the set of edges outgoing from $\mv_n$.
Then $\|u_n\|_{\ell^\infty_\mu}\le \tilde{\mu}$ for all $n\in \mathbb N$, but 
\begin{eqnarray*}
\|\mathcal I^+ u_n\|_{\ell^\infty_{\nu}} &=& \sup_{\mv \in \mV} \left| \sum_{\me \in \mE}\iota^+_{\mv\me} u_n(\me)\right|\nu(\mv)\\
 &\ge& \sum_{\me \in \mE^+_n} \left( \iota^+_{\mv_n\me} \mu(\me) \right)\frac{\tilde{\nu}}{\tilde{\mu}}\\
 &= &{\rm deg}^+(\mv_n)\frac{\tilde{\nu}}{\tilde{\mu}} \\
 &\ge &n\frac{\tilde{\nu}^2}{\tilde{\mu}},
\end{eqnarray*}
hence $\lim_{n\to \infty}\|\mathcal I^+ u_n\|_{\ell^\infty_{\nu}}=+\infty$. The remaining assertions can be proved likewise.
\end{proof}

\begin{rem}\label{rem:bije}
(1) Recall that the oriented incidence matrix $\mathcal I$ of a finite graph $\mG$ with $\kappa$ connected components has rank $|\mV|-\kappa$, see e.g.\ \cite[\S~8.3.1]{GodRoy01}. Hence, $\mathcal I$ is never surjective (resp., $\mathcal I^T$ is never injective). On the other hand, $\mathcal I$ is injective (resp., $\mathcal I^T$ is surjective) if and only if it has rank $|\mE|$, which is the case exactly when $\mG$ is a forest. More generally, $\mathcal I^T$ is surjective whenever $\mG$ is a uniformly locally finite forest.

(2) Let now $\mG$ be infinite and connected. If $f\in \mathbb R^\mV$, $f\not\equiv 0$, satisfies $\mathcal I^Tf=0$, then necessarily $f$ is constant and hence does not belong to $\ell^p_d(\mV)$ unless $\|{\bf 1}\|_{\ell^p_d}<\infty$, i.e., unless $\mG$ has finite surface (for $p<\infty$) or unless $\nu$ is bounded (for $p=\infty$).
\end{rem}

Taking into account Remark~\ref{rem:bije}, a direct computation yields the following.

\begin{cor}\label{cor:bije}
Let $p\in [1,\infty)$. Then a necessary condition for the operator $\mathcal I^T:\ell^p_d(\mV)\to \ell^p_a(\mE)$ to be an isomorphism is that $\mG$ be infinite; a sufficient one is that $\mG$ be a uniformly locally finite tree with infinite surface.
\end{cor}

Following~\cite{Yam77}, we consider an energy functional\footnote{ For the sake of notational simplicity, we stipulate that $\psi(\me)$ will be written as (indifferently) either $\psi(\mv,\mw)$ or $\psi(\mw,\mv)$, whenever $\me$ is an edge with endpoints $\mv,\mw$ and $\psi$ is a function from $\mE$ to $\mathbb R$.}
\[
\Efun_p:f\mapsto\frac{1}{p}\|\mathcal I^T f\|^p_{\ell^p_a}=\frac{1}{p}\sum_{\me \in \mE}a(\me) |f(\me_+)-f(\me_-)|^p,\qquad f\in \mathbb R^\mV.
\]
The definition of $\Efun_p$ is clearly independent of the orientation of $\mG$.
 In view of the mentioned analogies between discrete and continuous calculus $\Efun_p$ is the discrete companion to the functional introduced in~\eqref{plaplfunctcont}.
 {
We are particularly interested in its two versions $\Efun^N_p$ and $\Efun^D_p$ defined by
$$\Efun^N_p(f):=\left\{
\begin{array}{ll}
\Efun_p(f) & \hbox{if } f\in w^{1,p,2}_{a,d}(\mV),\\
+\infty & \hbox{if } f\in \ell^2_d(\mV)\setminus w^{1,p,2}_{a,d}(\mV),
\end{array}\right.\qquad 
\Efun^D_p(f):=\left\{
\begin{array}{ll}
\Efun_p(f) & \hbox{if } f\in \mathring{w}^{1,p,2}_{a,d}(\mV),\\
+\infty & \hbox{if } f\in \ell^2_d(\mV)\setminus \mathring{w}^{1,p,2}_{a,d}(\mV).
\end{array}\right.$$
}

\begin{rem}\label{rem:electric}
A natural analog of our setting arises in the theory of electric circuits, following the scheme
\begin{itemize}
\item $u(\mv)$ $\to$ voltage at $\mv$,
\item ${\mathcal I}^T u(\me)$ $\to$ potential difference between the endpoints of $\me$,
\item $a(\me)$ $\to$ conductance of the branch corresponding to $\me$,
\item $a(\me){\mathcal I}^T u(\me)$ $\to$ current flowing between the endpoints of $\me$,
\item setting Dirichlet boundary conditions on a certain set $\mV_0$ of nodes $\to$ grounding the nodes in $\mV_0$,
\item replacing the values of a function $u$ in two nodes $\mv,\mw\in \mV$ by their average $\to$ shorting the nodes of the corresponding electric circuit,
\item replacing the value of $a$ in an edge $\me$ by 0 $\to$ cutting the branch of the corresponding electric circuit.
\end{itemize}
Observe for example that cutting branches of a finite electric circuit induces a shift in the Rayleigh-type quotient
\[
\frac{\Efun_p(f)}{\|f\|_{\ell^p_d}^p},
\]
which
is the natural object to minimize if one looks for eigenvalues, cf.~\cite{BuhHei09} (i.e., for those $\lambda\in \mathbb R$ s.t.\ 
\[
\lambda |\varphi(\mv)|^{p-2}\varphi(\mv)=\Lfun_p \varphi(\mv),\qquad \mv\in \mV,
\]
has a non-trivial solution $\varphi$). This shift is in accordance with the intuition that diffusion is slower on a graph with fewer edges and corresponds to Lord Rayleigh's Monotonicity Law (\cite[\S~1.4]{DoySne84}). We will not go into details of this interplay and refer instead to
 to the beautiful booklet~\cite{DoySne84}, where the linear theory is explained in detail.
\end{rem}

For $a\equiv 1$, $\Efun_1$ is the energy associated with the discrete mean curvature operator discussed e.g.\ in~\cite{ZhoSch05}, while $\Efun_2$ agrees with the quadratic form associated with~\eqref{kirchlapl}, and in fact the subdifferential of $\Efun_2$ is just the discrete Laplacian. If $\mG$ is uniformly locally finite, then by Lemma~\ref{lemma:car11} $\mathcal I$ and hence $\Delta^N_2$ are bounded linear operators. Accordingly, the initial value problem associated with
\[
\dot{\varphi}(t,\mv)=\Delta^N_2 \varphi(t,\mv),\qquad t\in\mathbb R,\; \mv \in\mV,
\]
is well-posed in $\ell^2_d(\mV)$ and the exponential matrices $e^{t\Delta^N_2}$, $t\in \mathbb R$, yield its solutions. 
Fiedler has showed that the largest eigenvalue of $\Delta^N_2$ is larger than the maximal node degree of $\mG$. Thus, dropping the assumption of uniform local finiteness $\Delta^N_2$ becomes an unbounded operator whose spectrum is not contained in any left half-plane, hence not the generator of a $C_0$-semigroup on $\ell^2_d(\mV)$. Still, it can be proved by methods based on quadratic forms that $-\Delta^N_2$ is instead a generator. Our aim in the next section is to show that, just like in the special case of $p=2$, most properties of the $p$-Laplacian {in the continuum} carry over to its discrete counterpart, despite the latter is non-local.

\section{Well-posedness results}\label{sec:3}

{
\begin{defi}\label{defi->thm:main}
The operator $\Lfun^N_p$ is defined by
\begin{eqnarray*}
D(\Lfun^N_p)&=&\bigg\{f\in w^{1,p,2}_{a,d}(\mV):\exists g\in \ell^2_d(\mV)\\
&&\qquad \hbox{ s.t.\ }\sum_{\me\in \mE} a(\me)|({\mathcal I^T}f)(\me)|^{p-2}({\mathcal I^T}f)(\me) ({\mathcal I^T}h)(\me)=\sum_{\mv \in \mV} g(\mv)h(\mv)d(\mv)\;\; \forall h \in w^{1,p,2}_{a,d}(\mV)\bigg\},\\
\Lfun^N_p f&=&g.
\end{eqnarray*}
Let $p\in (1,\infty)$. If $T>0$, $f_0\in w^{1,p,2}_{a,d}(\mV)$ and $f\in L^2(0,T;\ell^2_d(\mV))$, then a \emph{solution} of the the Cauchy problem 
\begin{equation}
\tag{HE$^N_p$}
\left\{
\begin{array}{rcll}
\dot{\varphi}(t,\mv)&=&-\Lfun^N_p \varphi(t,\mv)+f(t),\qquad &t\in [0,T],\; \mv\in \mV,\\
\varphi(0,\mv)&=&f_0(\mv),&\mv\in \mV,
\end{array}
\right.\end{equation}
is a function $\varphi\in H^1(0,T;\ell^2_d(\mV))\cap L^\infty(0,T;w^{1,p,2}_{a,d}(\mV))$ for which
\begin{itemize}
\item $\varphi(t,\cdot)\in D(\Lfun^N_p)$ for a.e.\ $t\in [0,T]$,
\item $\dot{\varphi}(t,\mv)=-\Lfun^N_p \varphi(t,\mv)+f(t)$ is satisfied for a.e.\ $t\in [0,T]$, and 
\item $\varphi(0,\cdot)=f_0$.
\end{itemize}
\end{defi}
The above operator is defined weakly on the largest possible subspace of $\ell^2_d(\mv)$. Therefore, in analogy with the case of the $p$-Laplacian in the continuum and inspired by~\cite{HaeKelLen12}, we think of it as a discrete version of the Neumann $p$-Laplacian, whence the notation. Similarly, we introduce a discrete analog of the Dirichlet $p$-Laplacian.
\begin{defi}\label{defi->thm:main_D}
The operator $\Lfun^D_p$ is defined by
\begin{eqnarray*}
D(\Lfun^D_p)&=&\bigg\{f\in \mathring{w}^{1,p,2}_{a,d}(\mV):\exists g\in \ell^2_d(\mV)\\
&&\qquad \hbox{ s.t.\ }\sum_{\me\in \mE} a(\me)|({\mathcal I^T}f)(\me)|^{p-2}({\mathcal I^T}f)(\me) ({\mathcal I^T}h)(\me)=\sum_{\mv \in \mV} g(\mv)h(\mv)d(\mv)\;\; \forall h \in \mathring{w}^{1,p,2}_{a,d}(\mV)\bigg\},\\
\Lfun^D_p f&=&g.
\end{eqnarray*}
Let $p\in (1,\infty)$. If $T>0$, $f_0\in \mathring{w}^{1,p,2}_{a,d}(\mV)$ and $f\in L^2(0,T;\ell^2_d(\mV))$, then a \emph{solution} of the the Cauchy problem 
\begin{equation}
\tag{HE$^D_p$}
\left\{
\begin{array}{rcll}
\dot{\varphi}(t,\mv)&=&-\Lfun^D_p \varphi(t,\mv)+f(t),\qquad &t\in [0,T],\; \mv\in \mV,\\
\varphi(0,\mv)&=&f_0(\mv),&\mv\in \mV,
\end{array}
\right.\end{equation}
is a function $\varphi\in H^1(0,T;\ell^2_d(\mV))\cap L^\infty(0,T;\mathring{w}^{1,p,2}_{a,d}(\mV))$ for which
\begin{itemize}
\item $\varphi(t,\cdot)\in D(\Lfun^D_p)$ for a.e.\ $t\in [0,T]$,
\item $\dot{\varphi}(t,\mv)=-\Lfun^D_p \varphi(t,\mv)+f(t)$ is satisfied for a.e.\ $t\in [0,T]$, and 
\item $\varphi(0,\cdot)=f_0$.
\end{itemize}
\end{defi}
In general $\mathcal L^N_p$ and $\mathcal L^D_p$ do not agree, unless $w^{1,p,2}_{a,d}(\mV)=\mathring{w}^{1,p,2}_{a,d}(\mV)$ -- e.g., if $\mG$ is finite. In this case, $\rm{(HE^N_p)}$ and $\rm{(HE^D_p)}$ coincide with the initial value problem for a finite dimensional dynamical system $\rm{(HE_p)}$ whose solution is given by Carathéodory's theorem (see e.g.~\cite[Thm.~2.10]{ChiFas10}). 
}

{
Our main result of this section states that both  $\rm{(HE^N_p)}$ and $\rm{(HE^D_p)}$ are well-posed also if $\mG$ is infinite: The proof is a direct application of the theory of semigroups generated by $m$-accretive operators. We can  also obtain an alternative proof of well-posedness of  $\rm{(HE^D_p)}$ exploiting the Galerkin scheme described in Theorem~\ref{theo:galerkinsch}. The interesting and perhaps novel aspect of this alternative proof is that it can be interpreted as a convergence assertion for a sequence of solutions of certain $p$-heat equations on \emph{finite} graphs.}

 {
 To explain this point, we first need to make precise what we understand under \emph{graph convergence}. There are many related notions in the literature, the most popular being possibly those proposed by Benjamini and Schramm in~\cite{BenSch01} and by Lovász and Szegedy in~\cite{LovSze06}.
 The latter approach could probably be adapted to our setting, but we prefer to simply define a notion of convergence that is already tailored for our needs.}
 
 {
Let $\mG$ be a graph with node set $\mV$ and edge set $\mE$ and consider a node subset $\mW\subset \mV$. We recall that the \emph{subgraph $\mG[\mW]$ of $\mG$ induced by $\mW$} is by definition the graph whose node set is $\mW$ and whose edge set consists of all edges in $\mE$ whose endpoints are both in $\mW$. If $\mG$ is a \emph{weighted} graph with edge weights $\mu$ and node weights $\nu$, then each edge and each node in the induced subgraph keeps the same weight it has in $\mG$.
\begin{defi}
Let $\mG$ be a graph. A family of graphs $(\mG_n)_{n\in \mathbb N}$ is called \emph{growing} if $\mG_n$ is an induced subgraph of $\mG_m$ for all $n,m\in \mathbb N$ with $n\le m$. It is said to \emph{exhaust $\mG$} if
\begin{itemize}
\item $\mG_n$ is an induced subgraph of $\mG$  for all $n\in \mathbb N$  and 
\item $\bigcup_{n\in \mathbb N}\mV_n =\mV$.
\end{itemize}
\end{defi}
We are finally in the position to state our main result.
\begin{theo}\label{thm:main}
Let $p\in (1,\infty)$, $T>0$, and $f\in L^2(0,T;\ell^2_d(\mV))$. Then the following assertions hold.
\begin{enumerate}[(1)]
\item The Cauchy problem $\rm(HE^N_p)$ admits a unique solution for all initial data $f_0\in w^{1,p,2}_{a,d}(\mV)$. 
\item The Cauchy problem $\rm(HE^D_p)$ admits a unique solution for all initial data $f_0\in \mathring{w}^{1,p,2}_{a,d}(\mV)$.
\item   If $f=0$, then the solution to either problem is given by a $C_0$-semigroup of nonlinear contractions.
\item More precisely, for all $f_0\in \mathring{w}^{1,p,2}_{a,d}(\mV)$ there is a growing family of finite graphs $(\mG_n)_{n\in \mathbb N}$ that exhausts $\mG$ and such that the sequence of solutions $(\varphi_n)_{n\in \mathbb N}$ to the Cauchy problem 
\begin{equation}
\tag{HE$^{(n)}_p$}
\left\{
\begin{array}{rcll}
\dot{\varphi}_n(t,\mv)&=&-\Lfun^{(n)}_p \varphi_n(t,\mv)+f(t),\qquad &t\in [0,T],\; \mv\in \mV_n,\\
\varphi_n(t,\mv)&=&0, &t\in [0,T],\; \mv \in \mV\setminus\mV_n,\\
\varphi_n(0,\mv)&=&f_0(\mv), &\mv\in \mV_n,
\end{array}
\right.\end{equation}
where $\mV_n$ denotes the node set of $\mG_n$ and for all $h\in c_{00}(\mV)$
\begin{eqnarray*}
\Lfun^{(n)}_p h(\mv)&:=&
\left\{
\begin{array}{ll}
\frac{1}{d(\mv)}\sum\limits_{\substack{\mw\in \mV_n\\ \mw \sim \mv}} a(\mv,\mw)\left|h(\mv)-h(\mw)\right|^{p-2}\left(h(\mv)-h(\mw)\right)\\
 \qquad + \frac{1}{d(\mv)}\left|h(\mv)\right|^{p-2}h(\mv)\sum\limits_{\substack{\mw\not\in \mV_n\\ \mw \sim \mv}} a(\mv,\mw),\qquad & \mv \in \mV_n,\\
 \\
-\frac{1}{d(\mv)}\left|f(\mv)\right|^{p-2}\left(f(\mv)\right)\sum\limits_{\substack{\mw\in \mV_n\\ \mw \sim \mv}} a(\mv,\mw),&\mv\not\in \mV_n,
\end{array}
\right.
\end{eqnarray*}
converges to the unique solution $\varphi$ of $\rm(HE^D_p)$, weakly in $H^1(0,\infty;\ell^2_d(\mV))$ and weakly$^*$ in $L^\infty(0,\infty;\mathring{w}^{1,p,2}_{a,d}(\mV))$. 
\item If additionally $f\equiv 0$ and $f_0\in \ell^q_d(\mV)$ for some $q\in [1,2)$ (resp., $f_0\in c_{0d}(\mV)$), then $(\varphi_n)_{n\in \mathbb N}$ converges to $\varphi$ also weakly$^*$ in $L^\infty(\mathbb R_+;\ell^q_d(\mV))$ (resp., in $L^\infty(\mathbb R_+;c_{0d}(\mV))$).
\end{enumerate}
\end{theo}
}

Before proving this theorem, we will need several preliminary results. We refer to the Appendix for notations and terminology.



\begin{lemma}\label{propfunc}
Let $p\in (1,\infty)$. Then the functional $\Efun_p$ is convex and $\Efun_p+\frac{\omega}{2} \|\cdot \|_{\ell^2_d}^2$ is coercive for any $\omega\ge 0$. Furthermore, $\Efun^N_p$ is continuously Fréchet differentiable as a functional on $w^{1,p,2}_{a,d}(\mV)$ and lower semicontinuous as a functional on $\ell^2_d(\mV)$. { Likewise,  $\Efun^D_p$ is continuously Fréchet differentiable as a functional on $\mathring{w}^{1,p,2}_{a,d}(\mV)$ and lower semicontinuous as a functional on $\ell^2_d(\mV)$.}
\end{lemma}

\begin{proof}
Since for all $p\in (1,\infty)$
\begin{equation}\label{eq-ep}
\Efun_p \equiv\frac{1}{p}\|\cdot\|^p_{\ell^p_a(\mE)}\circ \mathcal I^T \qquad \hbox{on }w^{1,p,2}_{a,d}(\mV),
\end{equation}
the composition of a convex functional and a linear operator, the functional $\Efun_p$ is convex -- hence so is any perturbation by another convex functional, and in particular $\Efun_p(\cdot)+\frac{\omega}{2} \|\cdot \|_{\ell^2_d}^2$ is convex and coercive for any $\omega\ge 0$. 

In view of~\eqref{eq-ep}, and because $\mathcal I^T$ is bounded 	from $w^{1,p,2}_{a,d}(\mV)$ to $\ell^p_a(\mE)$, in order to check continuous Fréchet differentiability of $\Efun_p$ it suffices to observe that the functional $\|\cdot\|^p_{\ell^p_a}$ on $\ell^p_a(\mE)$ is continuously Fréchet differentiable for $p\in (1,\infty)$ -- with  
\begin{eqnarray*}
\Efun^{N'}_p(f)h&=&\left\langle |{\mathcal I^T}f|^{p-2}{\mathcal I^T}f,{\mathcal I}^T h\right\rangle_{\ell^{p'}_a,\ell^p_a}\\
&=&\sum_{\me\in \mE} a(\me)|({\mathcal I^T}f)(\me)|^{p-2}({\mathcal I^T}f)(\me) ({\mathcal I^T}h)(\me),\qquad f,h\in w^{1,p,2}_{a,d}(\mV),
\end{eqnarray*}
by the chain rule.  Thus, $\Efun_p$ is in particular lower semicontinuous as a functional on $w^{1,p,2}_{a,d}(\mV)$, and lower semicontinuity as a functional on $\ell^2_d(\mV)$ follows from~\cite[Lemma~IV.5.2]{Sho97}. 

The same assertions hold for $\Efun^D_p$, since it is simply the restriction to $\mathring{w}^{1,p,2}_{a,d}(\mV)$.
\end{proof}

 By virtue of Lemma~\ref{Lemmanittka}, Lemmas~\ref{propfunc} and~\ref{propw12p} yield now that $\Efun^N_p$ has a subdifferential that is at most single-valued, and so does  $\Efun^D_p$. They can be expressed weakly as follows.
 
 {
\begin{lemma}\label{identlp}
Let $p\in (1,\infty)$. The subdifferential $\partial \Efun^N_p$ of the functional $\Efun^N_p$ agrees with the operator $\Lfun^N_p$ introduced in Definition~\ref{defi->thm:main}. Likewise, the subdifferential $\partial \Efun^D_p$ of the functional $\Efun^D_p$ agrees with the operator $\Lfun^D_p$ introduced in Definition~\ref{defi->thm:main_D}. The action of both operators on test functions agrees and it is given by
\begin{equation}
\label{L_p:precise}
\Lfun^N_pf(\mv)=\Lfun^D_pf(\mv)=\frac{1}{d(\mv)}\sum_{\substack{\mw\in \mV\\ \mw\sim \mv}}a(\mv,\mw)\left|f(\mw)-f(\mv)\right|^{p-2}\left(f(\mw)-f(\mv)\right),\qquad f\in c_{00}(\mV),\; \mv\in \mV.
\end{equation}
\end{lemma}
Observe that the series in $D(\Lfun^N_p)$ converge absolutely by Hölder's inequality, whereas the series in~\eqref{L_p:precise} actually consists of finitely many terms only.}
\begin{proof}
Let $f\in w^{1,p,2}_{a,d}(\mV)$. By virtue {of Lemma~\ref{Lemmanittka}, $f$  is in the domain of $\partial\Efun^N_p $ and $g=\partial \Efun^N_p f$ if and only if
\[
\Efun^{N'}_p (f)h=\left(g|h\right)_{\ell^2_d}\quad \hbox{for all }h\in w^{1,p,2}_{a,d}(\mV),
\]
or rather
\[
 \sum_{\me\in \mE} a(\me)|({\mathcal I^T}f)(\me)|^{p-2}({\mathcal I^T}f)(\me) ({\mathcal I^T}h)(\me)=\sum_{\mv \in \mV} g(\mv)h(\mv)d(\mv)\quad \hbox{for all }h\in w^{1,p,2}_{a,d}(\mV).
\]
A similar argument holds for $\Efun^D_p$. } More explicitly, for all $\mv \in \mV$ and all $f \in c_{00}(\mV)$
\begin{eqnarray*}
\partial \Efun^N_p f(\mv)=\partial \Efun^D_p f(\mv)&= &\frac{1}{d(\mv)}\sum_{\me\in \mE} \iota_{\mv\me} \left( a(\me)\left| \sum_{\mw \in\mV}\iota_{\mw\me}f(\mw)\right|^{p-2}\left( \sum_{\mw \in\mV}\iota_{\mw\me}f(\mw)\right)\right)\\
&= &\frac{1}{d(\mv)}\sum_{\me\in \mE} \iota_{\mv\me} \left( a(\me)\left|f(\me_+)-f(\me_-)\right|^{p-2}\left(f(\me_+)-f(\me_-)\right)\right)\\
&= &\frac{1}{d(\mv)}\sum_{\substack{\mw\in \mV\\ \mw\sim \mv}}a(\mv,\mw)\left|f(\mw)-f(\mv)\right|^{p-2}\left(f(\mw)-f(\mv)\right).\end{eqnarray*}
This completes the proof.
\end{proof}

\begin{rem}\label{rem:identlp}
{
(1) In particular, the (discrete) $p$-Laplace--Beltrami-type operator associated with the metric of $\mG$ is given by
\begin{equation}\label{beltr-eq}
{\Delta}_p f:=\frac{1}{\nu}\mathcal I(\mu|\mathcal I^T f|^{p-2}\mathcal I^T f),\qquad f\in c_{00}(\mV),
\end{equation}
in a variational sense, and for $\mu\equiv 1$ and $\nu\equiv 1$ we find 
\begin{equation*}\label{plapl1}
\Delta_{p} f={\mathcal I}\left( |{\mathcal I}^T f|^{p-2}{\mathcal I}^T f\right),\qquad f\in c_{00}(\mV).
\end{equation*}
In view of the mentioned analogies between discrete and continuous calculus, this is the correct discrete pendant of the $p$-Laplacian in the continuum.
}

(2) If $\mG$ is uniformly locally finite, then $\Efun^N_p$ and $\Efun^D_p$ are lower semicontinuous also for $p=1$. However, we cannot invoke again~\cite[Lemma~IV.5.2]{Sho97}, as this result relies upon reflexivity of the domain of the functional. Nevertheless, by Lemma~\ref{lemma:car11} $\mathcal I^T $ is a bounded linear operator from $\ell^2_d(\mV)$ to $\ell^2_a(\mE)$. Moreover, lower semicontinuity of $\|\cdot\|_{\ell^1_a(\mE)}$ (viewed as a function from $\ell^2_a(\mE)$ to $[0,+\infty]$) follows from Fatou's Lemma. In view of
\[
\Efun^N_1 \equiv \|\cdot\|_{\ell^1_a}\circ \mathcal I^T:\ell^2_d(\mV)\to [0,+\infty],
\]
we also deduce lower semicontinuity of $\Efun^N_1$ and of its restriction $\Efun^D_1$, hence we may consider their subdifferentials, just like we have done in the case of $p>1$. 
In order to avoid technicalities when determining these subdifferentials, however, we will from now on restrict {attention} to the case of $p>1$.
\end{rem}

When studying qualitative properties of solutions to the $p$-heat equation, in the continuum most proofs are made easy by the locality of the operator. In the discrete case locality fails to hold: The key point in our proofs will then be that the considered orthogonal projections map vectors in $\ell^2_d(\mV)$ into vectors with smaller oscillation. 

We will need the following apparent property of the $\ell^p$-norm.

\begin{lemma}\label{almostconv}
{For all $k\in \mathbb R$ and $p\in [1,\infty)$ the function}
\[
f_{k,p}:[0,\infty)\ni \alpha\mapsto |k+\alpha|^p +|k-\alpha|^p\in [0,\infty)
\]
is strictly monotonically increasing.
\end{lemma}

In the special case of $p=2$, the following results have been obtained in~\cite{KelLen09} already. In the nonlinear setting, the notion of submarkovian semigroups has been proposed in~\cite{CipGri03}. Recall that a semigroup on a Banach space $X$ is called \emph{irreducible} if the only ideals of $X$ that are left invariant under it are the trivial ones.

{
\begin{prop}\label{thm:cipgri}
Let $p\in (1,\infty)$. Both nonlinear $C_0$-semigroups $(e^{-t\Lfun^N_p})_{t\ge 0}$  and  $(e^{-t\Lfun^D_p})_{t\ge 0}$ are submarkovian, i.e., they are order preserving and $\|\cdot\|_{\ell^\infty_d}$-contractive. Furthermore, each of them is irreducible if and only if $\mG$ is connected.
\end{prop}
}

 The proof of the assertion concerning irreducibility is due to René Pröpper.
 
In the special case $p=2$, by~\cite[Thm.~2.9 and Def.~2.8]{Ouh05} connectedness hence implies (through positivity and irreducibility) that if $f\ge 0$ but $f\not=0$, then $e^{-t\Delta^N_{2}}f>0$: This property of the discrete Laplacian is already known, see e.g.~\cite[Cor.~2.9]{KelLen09}.

\begin{proof}
The condition corresponding to~\eqref{cipgrieq1} for $\Efun^N_p$, i.e.,
\begin{equation*}
\Efun_p (f \wedge g) + \Efun_p (f \vee g) \leq \Efun_p (f) + \Efun_p (g)\quad\hbox{for all } f,g \in w^{1,p,2}_{a,d}(\mV),
 \end{equation*}
is sufficient (and necessary) for $(e^{-t\Lfun^N_p})_{t\ge 0}$ to be order preserving, by Proposition~\ref{Bar96}. If we can prove this condition to be satisfied, then in particular
\begin{equation*}
\Efun_p (f \wedge g) + \Efun_p (f \vee g) \leq \Efun_p (f) + \Efun_p (g)\quad\hbox{for all } f,g \in \mathring{w}^{1,p,2}_{a,d}(\mV),
 \end{equation*}
hence also $(e^{-t\Lfun^D_p})_{t\ge 0}$ is order preserving.

In particular, the semigroups are order preserving if
\[
|\mathcal I^T (f\wedge g)(\me)|^p + |\mathcal I^T (f\vee g)(\me)|^p \le
|\mathcal I^T f(\me)|^p + |\mathcal I^T g(\me)|^p\qquad  \hbox{for all } \me \in \mE\hbox{ and all }f,g \in w^{1,p,2}_{a,d}(\mV).
\]
This can be proved taking $\me\in \mE$ and $f,g \in w^{1,p,2}_{a,d}(\mV)$ and dividing the four possible cases 
\begin{itemize}
\item $f(\me_+)\le g(\me_+)$ and $f(\me_-)\le g(\me_-)$, 
\item $g(\me_+)\le f(\me_+)$ and $g(\me_-)\le f(\me_-)$, 
\item $g(\me_+)\le f(\me_+)$ and $f(\me_-)\le g(\me_-)$,
\item $f(\me_+)\le g(\me_+)$ and $g(\me_-)\le f(\me_-)$. 
\end{itemize}
The assertion clearly holds in the first two cases. In order to check its validity in the third and fourth case, we prove that for all $x,y,w,z\in \mathbb R$ with $x\ge y$ and $w\ge z$
\[
|y-z|^p+|x-w|^p \le |x-z|^p + |y-w|^p,
\]
i.e.,
\[
|k+\alpha|^p+|k-\alpha|^p \ge |k+\beta|^p + |k-\beta|^p,
\]
where 
\[
k:=\frac{x+y-w-z}{2},\qquad \alpha:=\frac{x-y+w-z}{2},\qquad \beta:=\frac{-x+y+w-z}{2},
\]
In fact, if $x\ge y$ and $w\ge z$, then
\[
y-x+w-z\le y-x-w+z \le x-y-w+z,
\]
i.e., $|\alpha|>|\beta|$ and the assertion follows from Lemma~\ref{almostconv}. The other cases can be treated likewise.

Let us now prove that the semigroups are $\|\cdot\|_{\ell^\infty_d}$-contractive. First of all, because $\Efun_p$ is homogeneous (of degree $p$), by Proposition~\ref{Bar96} we have to prove the condition corresponding to~\eqref{cipgrieq2}, i.e.,  that for all $f,g \in w^{1,p,2}_{a,d}(\mV)$
\begin{eqnarray*}
&&\Efun_p \left(\frac{g+(f-g+1)_+}{2} + \frac{g-(f-g-1)_-}{2}\right)\\
&&\qquad\quad+\Efun_p \left(\frac{f-(f-g+1)_+}{2} + \frac{f+(f-g-1)_-}{2}\right)\\
&&\quad \le \Efun_p (f) + \Efun_p (g).
 \end{eqnarray*}
 As above, it suffices to take one $\me\in \mE$ and to check that for all $f,g \in w^{1,p,2}_{a,d}(\mV)$
\begin{eqnarray*}
&&\left|\mathcal I^T \frac{g+(f-g+1)_+}{2}(\me) + \mathcal I^T \frac{g-(f-g-1)_-}{2}(\me)\right|^p\\
&&\qquad\quad +\left|\mathcal I^T \frac{f-(f-g+1)_+}{2}(\me) + \mathcal I^T \frac{f+(f-g-1)_-}{2}(\me)\right|^p\\
&&\quad \le |\mathcal I^T f(\me)|^p + | \mathcal I^T g(\me)|^p.
 \end{eqnarray*}
or equivalently that  for all $x,y,w,z\in \mathbb R$
\begin{eqnarray*}
&&\left| \frac{y+(x-y+1)_+}{2}- \frac{z+(w-z+1)_+}{2} + \frac{y-(x-y-1)_-}{2}- \frac{z-(w-z-1)_-}{2}\right|^p\\
&&\qquad\quad +\left| \frac{x-(x-y+1)_+}{2}- \frac{w-(w-z+1)_+}{2} + \frac{x+(x-y-1)_-}{2}- \frac{w+(w-z-1)_-}{2}(\me_-)\right|^p\\
&&\quad \le | x-w|^p + | y-z|^p.
 \end{eqnarray*}
Proving this inequality in the nine possible cases
\begin{itemize}
\item $|x-y|\le 1$ and $|w-z|\le 1$, 
\item $|x-y|\le 1$ and $w-z\ge 1$, 
\item $|x-y|\le 1$ and $w-z\le -1$, 
\item $x-y\le -1$ and $|w-z|\le 1$, 
\item $x-y\le -1$ and $w-z\ge 1$, 
\item $x-y\le -1$ and $w-z\le -1$, 
\item $x-y\ge 1$ and $|w-z|\le 1$, 
\item $x-y\ge 1$ and $w-z\ge 1$, 
\item $x-y\ge 1$ and $w-z\le -1$, 
\end{itemize}
is tedious but not difficult. E.g., in the second case (the first being trivial) one has to check that
\[
\left| x- \frac{w-z+1}{2}\right|^p +\left|y- \frac{w+z-1}{2}\right|^p\le | x-w|^p + | y-z|^p.
\]
Also in this case, this condition can be re-written as
\[
|k+\alpha|^p+|k-\alpha|^p \ge |k+\gamma|^p + |k-\gamma|^p,
\]
where again
\[
k:=\frac{x+y-w-z}{2},\qquad \alpha:=\frac{-x+y+w-z}{2},
\]
and
\[
\gamma:=\frac{x-y-1}{2}.
\]
Under the assumption that ($|x-y|\le 1$ and) $w-z\ge 1$, one has 
\[
x-y-w+z\le x-y-1\le y-x+w-z, 
\]
and hence $|\gamma|<|\alpha|$, whence the assertion follows, by Lemma~\ref{almostconv}.

Finally, let us show the assertion about irreducibility. 
All ideals of $\ell^2_d(\mV)$ are of the form $\ell^2_d(\mV_0)\equiv \{f\in \ell^2_d(\mV): {\rm supp }
f\subset \mV_0\}$ for some subset $\mV_0\subset \mV$, and the associated orthogonal projections are given by the restriction operators $P_{\mV_0}:={\mathbf 1}_{\mV_0}\cdot$. 

Now, assume $\ell^2_d(\mV_0)$ to be invariant under $(e^{-t\Delta_{p}})$, or equivalently {(by Proposition~\ref{Bar96})} that
\begin{equation}\label{ouha_irred}
\Efun_p (P_{\mV_0}f)\le \Efun_p (f)\qquad \hbox{for all }f\in \ell^2_d(\mV).
\end{equation}
We have to show that $\mV_0$ is a trivial subset of $\mV$, i.e., $\mV_0 =\mV$ or $\mV_0=\emptyset$. 

In fact, if $\mV_0 \not=\mV\not=\mV_0^C$ there are two adjacent nodes $\mv_0\in \mV_0$ and $\mv_1\in \mV_0^C$.
Set
\[
\tilde{\mV}:=\mV\setminus\{\mv_0,\mv_1\},
\]
so that $\mE$ is partitioned into $(\mv_0,\mv_1),\mE^0,\mE^1,\tilde{\mE}$, where for $i\in \{0,1\}$ $\mE^{i}$ consist of those edges \emph{other than $(\mv_0,\mv_1)$} one of whose endpoints is $\mv_i$ (regardless of their orientation), and $\tilde{\mE}:=\mE\setminus((\mv_0,\mv_1)\cup \mE^0\cup \mE^1)$.  In other words,
\begin{eqnarray*}
p\Efun_p (g)&=&a (\mv_0,\mv_1)|g(\mv_0)-g(\mv_1)|^p+ \sum_{\substack{\mw\sim \mv_0 \\ \mv\neq \mv_1}} a (\mv_0,\mw)|g(\mv_0)-g(\mw)|^p\\
&&+\sum_{\substack{\mw\sim \mv_1 \\ \mv\neq \mv_0}} a (\mv_1,\mw)|g(\mv_1)-g(\mw)|^p + \sum_{\me \in \tilde{\mE}} a (\me)|g(\me_+)-g(\me_-)|^p\qquad  \hbox{for all }g\in w^{1,p,2}_{a,d}(\mV).
\end{eqnarray*}

Let now $f\in \ell^2_d(\mV)$ be defined by
\[
f(\mv):=
\left\{
\begin{array}{ll}
x,\qquad & \hbox{if }\mv= \mv_0,\\
1,& \hbox{if }\mv= \mv_1,\\
0,& \hbox{otherwise},
\end{array}
\right.
\]
for some $x>0$ to be determined later.
Accordingly, 
\[
P_{\mV_0} f(\mv):=
\left\{
\begin{array}{ll}
x,\qquad & \hbox{if }\mv=\mv_0,\\
0,& \hbox{otherwise}.
\end{array}
\right.
\]
Therefore,
\[
p\Efun_p (f)
=a (\mv_0,\mv_1)|x-1|^p+ \sum_{\substack{\mw\sim \mv_0 \\ \mv\neq \mv_1}} a (\mv_0,\mw)|x|^p+\sum_{\substack{\mw\sim \mv_1 \\ \mv\neq \mv_0}} a (\mv_1,\mw).
\]
(observe that the sums on the RHS are finite, because they are less then ${\rm deg}(\mv_0)$ and ${\rm deg}(\mv_1)$, respectively -- recall that $\mG$ is assumed to be locally finite) whilst
\[
p\Efun_p (P_{\mV_0} f)
=a (\mv_0,\mv_1)|x|^p+ \sum_{\substack{\mw\sim \mv_0 \\ \mv\neq \mv_1}} a (\mv_0,\mw)|x|^p.
\]
Accordingly, 
\[
p\Efun_p(f)-p\Efun_p(P_{\mV_0}f)=a (\mv_0,\mv_1)\left(|x-1|^p -|x|^p\right)+ \sum_{\substack{\mw\sim \mv_1 \\ \mv\neq \mv_0}} a (\mv_1,\mw)<0,
\]
choosing $x$ large enough, thereby contradicting~\eqref{ouha_irred}. This implies that indeed $\mV_0=\emptyset$ or {$\mV_0=\mV$}.


Conversely, each subspace of $\ell^2_d(\mV)$ consisting of functions over a connected component is apparently left invariant under $(e^{-t\Lfun^N_p})_{t\ge 0}$, {hence the semigroup would be reducible if $\mG$ contained more than one connected components}.
\end{proof}

\begin{rem}\label{beurden-restr}
{
Observe that our proof of Proposition~\ref{thm:cipgri} yields that the same assertions hold for the semigroup generated by (minus) the subdifferential associated with any restriction of $\Efun^N_p$ and $\Efun^D_p$. In particular, the restriction of both $\Efun^N_p$ and $\Efun^D_p$ to any induced subgraph is associated with a submarkovian semigroup.}

{
(Let us emphasize that the subdifferential of the restriction $\Efun^N_p$ to an induced subgraph is in general different from the restriction to the same induced subgraph of the subdifferential of $\Efun^N_p$, and the same holds for $\Efun^D_p$).}
\end{rem}

\begin{cor}\label{ralphwell-local}
Let $p\in (1,\infty)$. Both $C_0$-semigroups $(e^{-t\Lfun^N_p})_{t\ge 0}$ and $(e^{-t\Lfun^D_p})_{t\ge 0}$ extrapolate to a family of nonlinear semigroups on $\ell^q_d(\mV)$ for all $q\in [1,\infty]$ as well as on
\[
c_{0d}(\mV):=\{f:\mV\to \mathbb R:\forall \varepsilon>0 \; \exists \mW \subset \mV, |\mW|<\infty, \hbox{ s.t. }|f(\mv)|d(\mv)<\varepsilon\; \forall \mv\not\in \mW \}.
\]
\end{cor}

Formally speaking, Corollary~\ref{ralphwell-local} yields well-posedness of a dynamical system only if one is able to determine the generators of the extrapolated semigroups. In the linear case of $p=2$, this can be done using~\cite[\S~II.2.3]{EngNag00} on $c_{0d}(\mV)$ and on $\ell^p_q(\mV)$ for $q\in [1,2)$, and by duality on $\ell^q_d(\mV)$ for $q\in (2,\infty]$. In the nonlinear case of $p\neq 2$ this technique seems to fail -- this is the reason why Theorem~\ref{thm:main}.(5) is relevant.

\begin{proof}
By Proposition~\ref{thm:cipgri}, contractivity of $(e^{-t\Lfun^N_p})_{t\ge 0}$ with respect to the norm of $\ell^\infty_d(\mV)$ yields that the semigroup on $\ell^2_d(\mV)$ extends to a contractive semigroup on the closure of $\ell^2_d(\mV)$ in the $\ell^\infty_d(\mV)$-norm, i.e., in $c_{0d}(\mV)$. By duality we obtain a contractive semigroup on $\ell^1_d(\mV)$, and finally again by duality a contractive semigroup on $\ell^\infty_d(\mV)$. Now, the assertion follows applying the non-linear Riesz--Thorin-type interpolation theorem due to Browder, see~\cite[Thm.~3.6]{CipGri03} for a version tailored for our setting. 
\end{proof}

We can finally procede to the proof of our main result.

\begin{proof}[Proof of Theorem~\ref{thm:main}]
{
Having proved  in Lemmas~\ref{propw12p} and~\ref{propfunc} that both triples
\begin{itemize}
\item $w^{1,p,2}_{a,d}(\mV)$, $\ell^2_d(\mV)$, and $\Efun^N_p$, and
\item $\mathring{w}^{1,p,2}_{a,d}(\mV)$, $\ell^2_d(\mV)$, and $\Efun^D_p$
\end{itemize} 
satisfy for all $p\in (1,\infty)$ the Assumptions~\ref{assum:append}, the assertions in (1), (2), and (3) follow directly from Theorem~\ref{thm:brezkato}.
}

{
(4)  Consider the sequence of finite dimensional subspaces of $\ell^2_d(\mV)$ of the form 
\[
\ell^2_d(\mV_n)\equiv\{f:\in \ell^2_d(\mV): f(\mv_m)=0\; \forall m>n\},
\]
where $\mV_n:=\{\mv_1,\ldots,\mv_n\}$. Hence $\ell^2_d(\mV_n)\subset c_{00}(\mV)$ for all $n\in \mathbb N$: This defines a total sequence in $\mathring{w}^{1,p,2}_{a,d}(\mV)$. Let us introduce an orientation of $\mG$ as follows: an edge $\me=(\mv_i,\mv_j)$ is always outgoing from the node with lower index. Let us denote by $\mE_n$ the subset of $\mE$ of all those edges with both endpoints in $\mV_n$, and by $\mE_n^+$ the subset of $\mE$ of all those edges with exactly one endpoint in $\mV_n$ (this is sometimes called the \emph{boundary} of $\mV_n$ in graph theory).
In this way we can describe the subdifferential of $\Efun^{(n)}_p:=\Efun_{p|_{\ell^2_d(\mV_n)}}$. Indeed, a  computation similar to that in the proof of Lemma~\ref{identlp} yields for all $f\in \ell^2_d(\mV_n)$ 
\begin{eqnarray*}
\partial\Efun^{(n)}_p (f)(\mv)&=&\frac{1}{d(\mv)}\sum_{\me\in \mE} \iota_{\mv\me} \left( a(\me)\left|f(\me_+)-f(\me_-)\right|^{p-2}\left(f(\me_+)-f(\me_-)\right)\right)\\
&=&\frac{1}{d(\mv)}\sum_{\me\in \mE_n} \iota_{\mv\me} \left( a(\me)\left|f(\me_+)-f(\me_-)\right|^{p-2}\left(f(\me_+)-f(\me_-)\right)\right)+ \frac{1}{d(\mv)}\sum_{\me\in \mE_n^+} \iota_{\mv\me} \left( a(\me)\left|f(\me_+)\right|^{p-2}\left(f(\me_+)\right)\right)\\
&=&\frac{1}{d(\mv)}\sum_{\substack{\mw\in \mV_n\\ \mw \sim \mv}} a(\mv,\mw)\left|f(\mv)-f(\mw)\right|^{p-2}\left(f(\mv)-f(\mw)\right)+ \frac{1}{d(\mv)}\left|f(\mv)\right|^{p-2}\left(f(\mv)\right)\sum_{\substack{\mw\not\in \mV_n\\ \mw \sim \mv}} a(\mv,\mw)
\end{eqnarray*}
if $\mv\in \mV_n$, and 
\[
\partial\Efun^{(n)}_p (f)(\mv)=-\frac{1}{d(\mv)}\left|f(\mv)\right|^{p-2}\left(f(\mv)\right)\sum_{\substack{\mw\in \mV_n\\ \mw \sim \mv}} a(\mv,\mw)
\]
if instead $\mv\in \mV\setminus \mV_n$ (in particular, $\partial\Efun^{(n)}_p (f)(\mv)=0$ if $\mv$ is not adjacent to any node in $\mV_n$).
Hence, the associated Cauchy problem is $\rm(HE^{(n)}_p)$. Adapting to the present setting the usual arguments that yield convergence of the Galerkin scheme, cf.\ Theorem~\ref{theo:galerkinsch}, one sees that the sequence of solutions of $\rm(HE^{(n)}_p)$ converges (up to subsequences) to the solution found in (2) -- weakly in $H^1(0,T;\ell^2_d(\mV))$ and weakly$^*$ in $L^\infty(0,T;w^{1,p,2}_{a,d}(\mV))$.
}

{
(5) Let now $f\equiv 0$. Then for all $n\in \mathbb N$ the solution to each $\rm(HE^{(n)}_p)$ is given by a $C_0$-semigroup of nonlinear contractions. Because the restricted energy functional $\Efun_p^{(n)}$ is a proper, convex and lower semicontinuous functional, (minus) its subdifferential generates a semigroup on $\ell^2_d(\mV_n)$, which is $\|\cdot\|_{\ell^\infty_d}$-contractive by Proposition~\ref{thm:cipgri}. 
By the same arguments used in the proof of Corollary~\ref{ralphwell-local}, this semigroup extrapolates to $c_{0d}(\mV_n)$ as well as to $\ell^q_d(\mV_n)$ for all $q\in [1,\infty]$. 
In particular, for all $f_0\in \ell^q_d(\mV)$
\[
\|e^{-t\mathcal L^{(n)}_p} f_0\|_{\ell^q_d(\mV_n)}\le\|f_{0_{|\mV_n}}\|_{\ell^q_d(\mV_n)}\le \|f_0\|_{\ell^q_d(\mV)}\qquad \hbox{for all }t\ge 0,
\]
due to the fact that $\Efun_p^{(n)}(0)=0$ and hence $e^{-t\mathcal L^{(n)}_p}0=0$ (cf.\ Remark~\ref{onecannot}.(1)), and because of contractivity of the semigroup. Because for all $q\in (1,\infty)$ $L^\infty(0,\infty;\ell^q_d(\mV))=L^1(0,\infty;\ell^{q'}_d(\mV))'$ is the dual of a separable Banach space, by the theorem of Banach--Alaoglu for all $t>0$ the bounded sequence $(e^{-t\mathcal L^{(n)}_p} f_0)_{n\in \mathbb N}\subset L^\infty(0,\infty;\ell^q_d(\mV))$ has a weak$^*$-convergent subsequence. 
}
\end{proof}

\begin{rem}\label{onecannot}
 {(1) In view of Proposition~\ref{Bar96}, and because 
 \[
 ({\rm Id}+\lambda \partial \Efun_p)(\alpha f)=\alpha({\rm Id}+|\alpha|^{p-2}\lambda\partial \Efun_p) (f)\qquad \hbox{for all } \alpha \in \mathbb R,
 \] 
 one cannot in general expect the mappings $J_\lambda(\Efun_p)$ and hence $e^{-t\Lfun^N_p}$ or $e^{-t\Lfun^D_p}$  to be homogeneous of any degree, for any $p \neq 2$, any $\lambda>0$ and any $t>0$. However, $e^{-t\Lfun^N_p} =e^{-t\Lfun^D_p} 0=0$ because $\Efun_p(0)=0$. This shows in particular that one can plug $f=0$ in~\eqref{eq:order-pres} and deduce from Proposition~\ref{thm:cipgri} that for all $p\in (1,\infty)$ $(e^{-t\Lfun^N_p})_{t\ge 0}$ and $(e^{-t\Lfun^D_p})_{t\ge 0}$ are positivity preserving, too. }

(2) Furthermore, Proposition~\ref{thm:cipgri} yields the following comparison result:
Let $\lambda>0$ and $p\in (1,\infty)$. If $\varphi_1,\varphi_2$ are the solutions of the elliptic equation with inhomogeneous data $f_1,f_2$ and if $f_1\ge f_2$, then also the solutions $\varphi_1,\varphi_2$ of the respective elliptic problems of the type
\begin{equation}\label{elliptdeltap}
\lambda \varphi(\mv)+\Lfun^N_p \varphi(\mv)= f(\mv),\qquad \mv\in \mV,
\end{equation}
satisfy $\varphi_1\ge \varphi_2$. If in particular $f\ge 0$, then the solution $\varphi$ of~\eqref{elliptdeltap} satisfies $\varphi\ge 0$. These results are related to some comparison principles that are known to hold on finite unweighted graphs, cf.~\cite{JiaChuReg04}.

{
(3) Let $p\in  (1,\infty)$. The set of eigenvectors of $\Lfun^N_p $ for the eigenvalue 0 is a subspace of $\ell^2_d(\mV)$ spanned by the characteristic vectors of the connected component of $\mG$ with finite surface, in the sense of Definition~\ref{defi:basic}. In fact, if $f\in \ell^2_d(\mV)$ is an eigenfunction of $\Lfun^N_p $ with eigenvalue 0, then $\mathcal I^T f=0$. By Remark~\ref{rem:bije} this implies $f=0$ if and only if each connected component of $\mG$ has infinite surface.
}

{
(4) Let $\mG$ be finite, so that $(HE^D_p)=(HE^N_p)$. Applying~\eqref{eq:degiorgi} we find that for all $p>1$ and all $f_0\in w^{1,p,2}_{a,d}(\mV)$ the solutions to $\rm(HE_p)$ either converge in finite time towards a constant function (which by (3) in the finite case are the only ground states), or else have ever decreasing energy. Of course, constant functions are not interesting for applications -- say, to clustering problems. However, it is not clear whether one can reach another, non trivial state, perhaps associated to the second eigenvalue, as a different equilibrium of the discrete $p$-heat equation. The discussion in~\cite{BuhHei09} and in particular in~\cite[\S~3]{BuhHei10} seems to suggest that a nontrivial eigenvector of the $p$-Laplacian could be obtained studying the $p$-heat equation on some suitable submanifold. This topic will be discussed in a forthcoming paper.
}
\end{rem}

\section{Symmetries}\label{sec:symm}

Graphs bearing some symmetry are well-studied objects of graph theory. In this section we are going to show how  certain non-trivial symmetries of the discrete $p$-heat equation arise if the underlying graph enjoys special symmetry properties.  {For the sake of simplicity, in this section we state all results solely for the Neumann $p$-Laplacian. Furthermore, we assume throughout this section that 
\[
a\equiv\mu\quad \hbox{ as well as }\quad d\equiv\nu.
\]
(i.e., $\Lfun^N_p\equiv \Delta_p$ as in~\eqref{beltr-eq}).}

\begin{defi}
A permutation $O$ on $\mV$ is called a \emph{node automorphism} of $\mG=(\mV,\mE,a,d)$ if for all $\mv,\mw\in \mV$
\begin{itemize}
\item $d(O\mv)=d(\mv)$ and
\item the entries of the adjacency matrix $\mathcal A$ introduced in~\eqref{adjdefi} satisfy $\alpha_{O\mv\;O\mw}=\alpha_{\mv\mw}$ (i.e., $(O\mv,O\mw)\in \mE$ or $(O\mw,O\mv)\in \mE$ if and only if $(\mv,\mw)\in \mE$ or $(\mw,\mv)\in \mE$) and furthermore $a(O\mv,O\mw)=a(\mv,\mw)$ whenever $\alpha_{\mv,\mw}\neq 0$.
\end{itemize}
We denote by ${\rm Aut}(\mG)$ the group of all node automorphisms of $\mG$. 
\end{defi}

In the unweighted case, the above definition reduces to the usual one: Node automorphisms are permutations on $\mV$ that preserve the adjacency relation. This definition does not depend on the orientation of $\mG$.

Observe that each node automorphism $O$ induces a permutation $O_L$ on $\mE$ by
$$O_L \me=(O\mv,O\mw)\qquad \hbox{if }\me=(\mv,\mw).$$
Thus, each edge permutation $O$ defines a mapping on $\mathbb R^\mV$ and a mapping on $\mathbb R^\mE$ defined as the Nemitskii operators associated with $O$ and $O_L$, respectively, i.e., 
\[
f\mapsto f(O\cdot)\qquad \hbox{and}\qquad u\mapsto u(O_L\cdot).
\]
With a slight abuse of notation, we will not distinguish between the node/edge permutations and their associated Nemitskii operators.

\begin{theo}
Let $p\in (1,\infty)$. If $O\in {\rm Aut}(\mG)$, then
$$e^{-t\Delta_p}O^k = O^k e^{-t\Delta_p}\qquad\hbox{for all } t\ge 0\hbox{ and }k\in \mathbb N.$$
\end{theo}

If $\mG$ is unweighted, this result is clear for $p=2$, since then $\Delta_2=D-A$ for a diagonal matrix $D$.

\begin{proof}
Clearly, it suffices to prove the claimed commutation relation for $k=1$. This can be checked owing to Corollary~\ref{commut_lemma}, observing that setting $\Sigma:=O$ one has
\[
L=\frac12 I_{H_1}\qquad \hbox{and}\qquad R=\frac12 I_{H_2}.
\]
Indeed, for all $f,g\in w^{1,p,2}_{a,d}(\mV)$
\begin{eqnarray*}
&&\Efun_p\left(\frac{f+O^*g}{2}\right) +\Efun_p \left(\frac{Of+g}{2}\right)\\
&&\qquad=\frac{1}{p}\sum_{\me\in \mE} a(\me) \left|\frac{({\mathcal I}^T f)(\me)+({\mathcal I}^T O^*g)(\me)}{2}\right|^p+\frac{1}{p}\sum_{\me\in \mE} a(\me) \left|\frac{({\mathcal I}^T g)(\me)+({\mathcal I}^T Of)(\me)}{2}\right|^p.
\end{eqnarray*}
By a change of variables and using the fact that $a$ is constant along induced orbits we obtain
\begin{eqnarray*}
&&\Efun_p\left(\frac{f+O^*g}{2}\right) +\Efun_p \left(\frac{Of+g}{2}\right)\\
&&\qquad=\frac{1}{p}\sum_{O_L \me\in \mE} a(\me) \left|\frac{({\mathcal I}^T Of)(\me)+({\mathcal I}^T g)(\me)}{2}\right|^p+\frac{1}{p}\sum_{\me\in \mE} a(\me) \left|\frac{({\mathcal I}^T Of)(\me)+({\mathcal I}^T g)(\me)}{2}\right|^p\\
&&\qquad=\frac{1}{p}\frac{1}{2^{p-1}}\sum_{\me\in \mE} a(\me) \left|({\mathcal I}^T Of)(\me)+({\mathcal I}^T g)(\me)\right|^p\\
&&\qquad\le\frac{1}{p}\left(\sum_{\me\in \mE} a(\me) \left(\left|({\mathcal I}^T Of)(\me)|^p+|({\mathcal I}^T g)(\me)\right|^p\right)\right)\\
&&\qquad=\frac{1}{p}\left(\sum_{\me\in \mE} a(\me) \left(\left|({\mathcal I}^T f)(\me)|^p+|({\mathcal I}^T g)(\me)\right|^p\right)\right)\\
&&\qquad=\Efun_p(f)+\Efun_p(g),
\end{eqnarray*}
as we wanted to prove.
\end{proof}

\begin{rem}\label{rem:short}
(1) Let $\mG$ have finite surface. Averaging a function $u$ over all nodes (i.e., shorting all nodes, in the point of view of Remark~\ref{rem:electric}) one obtains a system which is trivially left invariant under the evolution of the $p$-heat equation, by Proposition~\ref{Bar96} and because $\Efun_p({\mathbf 1})= 0$. This is independent of the automorphism group of $\mG$, see Remark~\ref{rem:expl}.(2) below.

(2) On the other hand, just shorting two arbitrary nodes is not sufficient to obtain an invariant subsystem: this can be easily seen by taking a path of length 3
with $a\equiv 1$ and considering the projection $P$ of $\ell^2(\mV)\equiv {\mathbb R}^4$ onto the space 
 \[
 \{f:\{\mv_1,\mv_2,\mv_3,\mv_4\}\to \mathbb R:f(\mv_2)=f(\mv_3)\}.
 \]
 Then, for $f(\mv_n):=n$ one has 
\[
2\left(\frac{3}{2}\right)^p=p\Efun_p(Pf)\not\le p\Efun_p(f)=3,\qquad p> 1,
\]
which by Proposition~\ref{Bar96} shows that said subspace is not invariant under $e^{-t\Delta_{p}}$ for any $t\ge 0$.

\end{rem}

Clearly, each subgroup of ${\rm Aut}(\mG)$ defines a partition of $\mV$ into equivalence classes with respect to their orbits. We denote by $[\mv]$ such orbits and by $|[\mv]|_d$ their lengths with respect to $d$, i.e.,
\[
|[\mv]|_d:=\sum_{\mw \in [\mv]} d(\mw)\equiv d(\mv)|[\mv]|.
\]
{Observe that the length $|[\mv]|_d$ is finite if and only if the set $[\mv]$ is finite. Finiteness of all orbits with respect to a subgroup of ${\rm Aut}(\mG)$ can be characterized e.g.~by means of~\cite[Lemma~1.27]{Woe00}.}

\begin{theo}\label{thm:orbit}
Let $p\in (1,\infty)$. Let $\Gamma$ be a subgroup of ${\rm Aut}(\mG)$ {each of whose orbits is a finite set}. Consider the orbitwise averaging operator $P$ defined by
\[
P f(\mv):=\frac{1}{|[\mv]|_d} \sum_{\mw\in [\mv]} f(\mw)d(\mw),\qquad f\in \ell^2_d(\mV),\; \mv\in\mV.
\]
Then the following assertions hold.
\begin{enumerate}[(1)]
\item The range of $P$ is left invariant under $(e^{-t\Delta^N_p})_{t\ge 0}$.
\item The null space of $P$ is left invariant under $(e^{-t\Delta^N_p})_{t\ge 0}$ if $p=2$.
\end{enumerate}
\end{theo}

\begin{exa}
1) A typical and relevant case is that of an infinite radial trees $\mT$, i.e., infinite rooted trees such that any two nodes with the same distance $n$ from the root have the same number $\gamma_n$ of children -- for the sake of simplicity, say, $\mT$ is binary, i.e., $\gamma_n=2$ for all $n\in \mathbb N$. Clearly, there are finitely nodes at distance $n$ from the root -- in fact, they are exactly $2^n$ if $\gamma_n=2$. 
Any orbitwise projection acts by simply averaging $\ell^2$-functions supported in an infinite binary subtree $\mT_0$ of $\mT$ over all nodes at the same distance from the root of $\mT_0$. Then, Theorem~\ref{thm:orbit} shows that radial initial data give rise to radial solutions. This is already known in the linear case $p=2$, both in the case of discrete and metric graphs (\cite{PicWoe89,Sol04}). The case of graphs that are not trees has been studied in~\cite{DutJor10}.

2) The range of $P$ with respect to a subgroup $\Gamma$ of ${\rm Aut}(\mG)$ is isomorphic to the node set of a new ``quotient graph'' $\mG/\Gamma$ obtained identifying all the nodes belonging to the same orbit; in the case of radial trees, if $\Gamma$ is the full automorphism group, then $\mG/\Gamma$ is a semi-infinite path, i.e., $\mathbb N$.
\end{exa}

Instead of proving the above theorem directly, we will deduce it as a corollary of a more general result.

\begin{defi}\label{defi:partition}
Let $I$ be a (possibly infinite) set. A partition 
\[
\mV=\dot{\bigcup_{i\in I}}\mV_i
\]
of the node set $\mV$ of $\mG$ is said to be \emph{almost equitable} with \emph{cells} $(\mV_i)_{i\in I}$ if for all $i,j\in I$, $i\neq j$,
\begin{equation}
\label{almostequi}
\hbox{ there are numbers }{c_{ij}\ge 0}\hbox{ s.t.\ }\sum_{\mw\in \mV_j} a(\mv,\mw)=c_{ij}d(\mv)\qquad \hbox{ for all }\mv \in \mV_i.
\end{equation}
We write
\[
|\mV_i|_d:=\sum_{\mv \in \mV_i} d(\mv),\qquad i\in I.
\]
\end{defi}

\begin{rem}\label{finalm}
Clearly, each node partition of a graph canonically induces an edge partition: Simply take \emph{edge} cells $\mE_{ij}$ consisting of all edges with initial endpoint in $\mV_i$ and terminal one in $\mV_j$ (observe that in general $\mE_{ij}\neq \mE_{ji}$). Moreover, we let $[\me]:=\mE_{ij}$ if $\me\in \mE_{ij}$ and write
\begin{equation}\label{defin:EA}
|\mE_{ij}|_a:=\sum_{\me \in \mE_{ij}} a(\me),\qquad i,j\in I.
\end{equation}
 
(1) Let $(\mV_i)_{i\in I}$ is an almost equitable partition. Then in particular, 
\[
|\mE_{ij}|_a=c_{ij}|{\mV_i}|_d\qquad \hbox{ for all }i,j\in I\hbox{ s.t.\ }i\neq j.
\]
In particular, if $|\mV_i|_d<\infty$ for each $i\in I$ (e.g., if $\mV$ has finite surface), then also $|\mE_{ij}|_a<\infty$ for each $i,j\in I$ s.t.\ $i\neq j$.

(2) If~\eqref{almostequi} holds for all $i,j$ (and not only for $i\neq j$), then the partition is called \emph{equitable}. 
Then, the \emph{quotient graph} is the directed, weighted multigraph with node set of cardinality $|I|$ (the $i$-th node $\mw_i$ corresponding to the cell $\mV_i$) such that $(\mw_i,\mw_j)$ is an edge (and if so, with weight $c_{ij}$) if and only if $c_{ij}\not=0$. A simple but useful result in the linear, unweighted, finite case is that the spectrum the discrete Laplacian of $\mG$ contains that of the discrete Laplacian of its quotient graph, see e.g.~\cite[Thm.~2.3]{Moh91}. In certain special cases the spectra even agree (\emph{not} counting multiplicity, of course), cf.~\cite[Thm.~7.8]{Chu97}.

(3) Definition~\ref{defi:partition} is a weighted generalization of the classical one for unweighted graphs ($a\equiv 1$, $d\equiv 1$), see e.g.~\cite{Moh91,God97} for equitable partitions and~\cite[\S~2.3]{BroHae12} for almost equitable ones. In the unweighted case, existence of an equitable node partition amounts to saying that each node in $\mV_i$ is adjacent to exactly $c_{ij}$ nodes in $\mV_j$.
\end{rem}

{
The following is the main result of this section. This kind of results seems to be new in the nonlinear case, but in the (linear) context of random walks on graphs these ideas have a long history. Indeed, shorting techniques have become first popular when they were proposed by Nash-Williams in order to prove recurrence of the random walk on the lattice $\mathbb Z^2$, cf.\ the exposition in~\cite[\S~2.2]{DoySne84}. A manifold of further potential theoretic problems can nowadays be treated by shorting methods, see e.g.~\cite[\S\S~2--3 and references therein]{Woe00}.
}

\begin{theo}\label{theo:equipart}
Let $p\in (1,\infty)$. Let $\mG$ have an {almost} equitable partition associated with cells $(\mV_i)_{i\in I}$ s.t.\ $|\mV_i|_d<\infty$ and $|\mE_{ij}|_a<\infty$ for all $i,j\in I$ {with $i\neq j$}.
Consider the cellwise averaging operator $P$ defined by
\[
P f(\mv):=\frac{1}{|\mV_i|_d} \sum_{\mw\in \mV_i} f(\mw)d(\mw),\qquad f\in \ell^2_d(\mV),\; \mv\in \mV_i.
\]
Then the following assertions hold.
\begin{enumerate}[(1)]
\item The range of $P$ is left invariant under $(e^{-t\Delta_p})_{t\ge 0}$.
\item The null space of $P$ is left invariant under $(e^{-t\Delta_p})_{t\ge 0}$ if $p=2$.
\end{enumerate}
\end{theo}

{In the linear case of $p=2$ invariance of both the range and the null space of $P$ under the semigroup amounts to saying that $P$ commutes with $(e^{-t\Delta^N_2})_{t\ge 0}$, hence 
Theorem~\ref{theo:equipart} can be seen as a weighted semigroup counterpart of a well-known result of algebraic graph theory, see e.g.~\cite[Thm.~9.3.3]{GodRoy01}; for the same reason, the sufficiency condition in~\cite[Thm.~1]{KelLenWoj11} is a corollary of the above theorem, for in a spherically symmetric graph the spheres clearly induce an equitable partition. Furthermore, Theorem~\ref{theo:equipart} is a discrete counterpart of the settings in~\cite{Sol04,CarMugNit08} (where different terminologies are used).}

\begin{proof}
By Proposition~\ref{Bar96}, the range of $P$ is invariant under $(e^{-t\Delta_p})_{t\ge 0}$ if and only if
\begin{equation}
\label{barthe1}
\|{\mathcal I}^T Pf\|^p_{\ell^p_a}\le \|{\mathcal I}^T f\|^p_{\ell^p_a}\qquad\hbox{for all } f\in w^{1,p,2}_{a,d}(\mv),
\end{equation}
whereas its null space is invariant if and only if
\begin{equation}
\label{barthe2}
\|{\mathcal I}^T (f-Pf)\|^p_{\ell^p_a}\le \|{\mathcal I}^T f\|^p_{\ell^p_a}\qquad\hbox{for all } f\in w^{1,p,2}_{a,d}(\mv).
\end{equation}

{
First of all, define for all $i,j\in I$, $i\neq j$, the averaging operator over the cell $\mE_{ij}$,
\[
\tilde{P}_{ij} u(\me):=\frac{1}{|\mE_{ij}|_a} \sum_{\mf\in \mE_{ij}} a(\mf)u(\mf),\qquad u\in \ell^2_a(\mE),\; \me\in \mE_{ij},
\]
which is an orthogonal projection on the Hilbert space $\ell^2(\mE_{ij})$.
}
Furthermore, observe that for all $f\in \ell^2_d(\mV)$, all $\me \in \mE$, and all $i,j\in I$
\begin{itemize}
\item  if $\me\in \mE_{ii}$, then ${\mathcal I}^T Pf(\me)=0$ since $Pf(\mv)=Pf(\mw)$ for all $\mv,\mw\in \mV_i$,
\item if $\me\in \mE_{ij}$ with $i\neq j$, and hence $\me_+\in \mV_i$ and $\me_-\in \mV_j$, then
\begin{eqnarray*}
{\mathcal I}^T Pf(\me)&=&Pf(\me_+)-Pf(\me_-)\\
&=&\frac{1}{|\mV_i|_d} \sum_{\mv\in \mV_i} f(\mv)d(\mv)-\frac{1}{|\mV_j|_d} \sum_{\mw\in \mV_j} f(\mw)d(\mw)\\
&=& \left(f\Big|d\left( \frac{{\mathbf 1}_{\mV_i}}{|\mV_i|_d}-\frac{{\mathbf 1}_{\mV_j}}{|\mV_j|_d}\right)\right)_{\ell^2(\mV)}\\
&\stackrel{(*)}{=}& \left(f\Big|\mathcal I \frac{a{\mathbf 1}_{\mE_{ij}}}{|\mE_{ij}|_a}\right)_{\ell^2(\mV)}\\
&=&\frac{1}{|\mE_{ij}|_a} \sum_{\mf \in \mE_{ij}}(\mathcal I^T f)(\mf)a(\mf)\\
&=&{\tilde{P}_{ij}\mathcal I^T f(\me)},
\end{eqnarray*}
where $(*)$ follows from the definition of {almost} equitable partition and Remark~\ref{finalm}.(2).
\end{itemize}

(1) In order to prove~\eqref{barthe1} for general $p>1$, we proceed as follows. One has in view of the above observations for all $f\in w^{1,p,2}_{a,d}(\mE)$
\begin{eqnarray*}
\sum_{\me\in\mE} a(\me)|\mathcal I^T Pf(\me)|^p
&=&\sum_{\substack{i,j \in I\\ i\neq j}}\sum_{\me \in \mE_{ij}} a(\me)|\mathcal I^T Pf(\me)|^p+\sum_{i \in I}\sum_{\me \in \mE_{ii}} a(\me)|\mathcal I^T Pf(\me)|^p\\
&=&\sum_{\substack{i,j \in I\\ i\neq j}} \sum_{\me \in \mE_{ij}} a(\me) |\tilde{P} \mathcal I^T f(\me)|^p\\
&=&\sum_{\substack{i,j \in I\\ i\neq j}} \sum_{\me \in \mE_{ij}} a(\me)\left|\frac{1}{|\mE_{ij}|_a} \sum_{\mf \in \mE_{ij}}a(\mf)\mathcal I^T f(\mf)\right|^p\\
&\stackrel{(*)}\le&\sum_{\substack{i,j \in I\\ i\neq j}} \sum_{\me \in \mE_{ij}} \frac{a(\me)}{|\mE_{ij}|_a} \sum_{\mf \in \mE_{ij}}a(\mf)\left|\mathcal I^T f(\mf)\right|^p\\
&=&\sum_{\substack{i,j \in I\\ i\neq j}}  \sum_{\mf \in \mE_{ij}}a(\mf)\left|\mathcal I^T f(\mf)\right|^p\\
& \le&\sum_{\me\in\mE} a(\me)|\mathcal I^T f(\me)|^p,
\end{eqnarray*}
where $(*)$ follows from Jensen's inequality. This shows that~\eqref{barthe1} holds and completes the proof.

{
(2) In order to check~\eqref{barthe2}, use again the above identities for $\mathcal I^T Pf$ to deduce that for all $f\in w^{1,p,2}_{a,d}(\mE)$
\begin{eqnarray*}
\|{\mathcal I}^T (f-Pf)\|^2_{\ell^2_a}&=&\sum_{\substack{i,j \in I\\ i\neq j}}\sum_{\me \in \mE_{ij}} a(\me)|\mathcal I^T f(\me)-\mathcal I^T Pf(\me)|^2+\sum_{i \in I}\sum_{\me \in \mE_{ii}} a(\me)|\mathcal I^T f(\me)-\mathcal I^T Pf(\me)|^2\\
&=&\sum_{\substack{i,j \in I\\ i\neq j}}\sum_{\me \in \mE_{ij}} a(\me)|\mathcal I^T f(\me)-\tilde{P}_{ij}\mathcal I^T f(\me)|^2+\sum_{i \in I}\sum_{\me \in \mE_{ii}} a(\me)|\mathcal I^T f(\me)|^2\\
&=&\sum_{\substack{i,j \in I\\ i\neq j}}\left\|\left({\rm Id}-\tilde{P}_{ij}\right)\mathcal I^T f\right\|^2_{\ell^2_a(\mE_{ij})}+\sum_{i \in I}\left\|\mathcal I^T f\right\|^2_{\ell^2_a(\mE_{ij})}\\
&\le &\sum_{\substack{i,j \in I\\ i\neq j}}\left\|\mathcal I^T f\right\|^2_{\ell^2_a(\mE_{ij})}+\sum_{i \in I}\left\|\mathcal I^T f\right\|^2_{\ell^2_a(\mE_{ij})}=\|{\mathcal I}^T f\|^2_{\ell^2_a},
\end{eqnarray*}
where the last inequality follows from the fact that each $\tilde{P}_{ij}$ is an orthogonal projection.
}
\end{proof}

\begin{rem}\label{rem:expl}
(1) There exist graphs on which the null space of $P$ is \emph{not} left invariant under $e^{-t\Delta_p}$ for any $p\neq 2$. An example for which~\eqref{barthe2} fails to hold for any $p\not=2$ is given as follows: Take $\mG$ to be a path of length 3, with weights $a\equiv 1$ and $d\equiv 1$. Consider the trivial equitable partition ($[\mv]:=\{\mv\}$ for all nodes) and the function 
\[
f\equiv \left(1,\frac12,0,x\right)\in \mathbb R^4,
\]
for $x\in \mathbb R$ to be fitted later. Then
\[
\mathcal I^T f\equiv \left(-\frac12,-\frac12,x\right),
\]
and choosing values of $x$ slightly larger than $1$ for $p>2$, and slightly smaller than $1$ for $p<2$ yields the sought-after counterexample, by elementary calculus arguments.

(2) Theorem~\ref{theo:equipart} is strictly more general than Theorem~\ref{thm:orbit}, since the orbit partition w.r.t.\ any subgroup $\Gamma$ of ${\rm Aut}(\mG)$ yields an equitable partition of $\mV$, but the converse is generally false (cf.~\cite{God97}) -- and \emph{a fortiori} there exist almost equitable partitions that do not derive from an orbit partition. 
\end{rem}

\begin{exa}
(1) As we have already observed in Remark~\ref{rem:short}.(1), averaging over all nodes of a graph with finite surface yields a projection whose range is invariant under $(e^{-t\Delta_p})_{t\ge 0}$ for any $p$. This is a particular instance of Theorem~\ref{theo:equipart}, since each graph has the (trivial) almost equitable partition given by $\mV_1\equiv \mV$. Furthermore, any graph has a trivial equitable partition -- simply take a partition each of whose cells is a singleton. However, neither of these partitions reduces complexity of the problem in any way. Hence, the natural question concerning Theorem~\ref{theo:equipart} is not whether it can be applied at all (indeed, it always can), but rather whether it can yield substantial information about the system. As a rule of thumb, the rougher is a non-trivial almost equitable partition we are able to find, the more interesting information we obtain.

(2) Let $\Gamma$ be a subgroup of ${\rm Aut}(\mG)$. The partitioning of $\mV$ into $Fix(\Gamma):=\{\mv \in \mV:Ov=v\;\forall O\in \Gamma\}$ and its complement is in general neither almost equitable, nor respected by $(e^{-t\Delta_p})_{t\ge 0}$ for any $p$, as the simple following example shows: Take $\mG_3$ as an unweighted $3$-star with a path of length 2 attached to the center, and consider the function $f$ defined by $f(v)= 0$ on each node $v$ of the star (including the center) and $f(v)=1$ on both the remaining nodes.
\end{exa}

\begin{exa}
{A prototypical class of infinite graphs with almost equitable (in fact, even equitable) partitions are the three regular tessellations of the plane, seen as unweighted graphs: They have equitable partitions given by the spheres, i.e., by the sets of nodes having same distance from one arbitrarily fixed root.  Among further classes of graphs that have (non-trivial!) almost equitable partitions, we mention so-called \emph{weakly spherically symmetric trees}, \emph{spherically symmetric trees}, \emph{antitrees}, and \emph{trees with complete spheres}. These classes have been thoroughly investigated in~\cite{BreKel11,KelLenWoj11}.}

{
It turns out that all these classes are particular instances of graphs with almost equitable partitions. Indeed, with our notation one can describe them as follows.
\begin{itemize}
\item Weakly spherically symmetric trees are possibly infinite graphs with a root $\mv_0$ and having an almost equitable partition that is induced by spheres, i.e., $\mV_i:=\{\mv \in \mV: {\rm dist }(\mv,\mv_0)=i\}$.
\item Spherically symmetric trees (sometimes called \emph{radial trees} in the literature; they are a generalization of \emph{homogeneous trees}, and in particular of \emph{binary trees}) are possibly infinite graphs with $a\equiv 1$, $d\equiv 1$, and having an  equitable partition that satisfies
\begin{itemize}
\item $|\mV_1|=1$,
\item $c_{ii}= 0$ for all $i\in I$,
\item $c_{ij}=0$ for all $i,j\in I$ such that $|j-i|\ge 2$, 
\item $c_{i-1\; i}|\mV_{i-1}|=|\mV_i|$ for all $i\in I$, and finally
\item $c_{i\; i-1}=1$ for all $i\in I$.
\end{itemize}
\item Antitrees are possibly infinite graphs with $a\equiv 1$, $d\equiv 1$, and having an equitable partition that satisfies
\begin{itemize}
\item $|\mV_1|=1$,
\item $c_{ii}= 0$ for all $i\in I$,
\item $c_{ij}=0$ for all $i,j\in I$ such that $|j-i|\ge 2$, 
\item $c_{i-1\; i}=|\mV_i|$ for all $i\in I$, and finally
\item $c_{i\; i-1}=|\mV_{i-1}|$ for all $i\in I$.
\end{itemize}
\item Trees with complete spheres (which are actually not trees, after all) are possibly infinite graphs with $a\equiv 1$, $d\equiv 1$, and having an equitable partition that satisfies
\begin{itemize}
\item $|\mV_1|=1$,
\item $c_{ii}= 0$ or $c_{ii}=|\mV_i|-1$ for all $i\in I$,
\item $c_{ij}=0$ for all $i,j\in I$ such that $|j-i|\ge 2$, 
\item $c_{i-1\; i}|\mV_{i-1}|=|\mV_i|$ for all $i\in I$, and finally
\item $c_{i\; i-1}=1$ for all $i\in I$.
\end{itemize}
\end{itemize}
Actually, in~\cite{BreKel11} the last three classes are mentioned as examples of a more general class -- the class  of what the authors call~\emph{path-commuting graphs}. }
\end{exa}
{It seems that in comparison with both weak spherical symmetry and path-commuting property, the notion of almost equitable partition permits a finer tuning when studying less symmetric graphs. Actually, we are not able to say what relationship exists between graphs with a non-trivial almost equitable partition and graphs with the other two properties. However, the following suggests that combining two different spherically symmetric graphs one can easily produce examples of graphs that are not path-commuting but do have non-trivial equitable partitions. Such ``combinations'' can reflect different graph operations.}
\begin{exa}
{(1) Consider a graph $\mG$ obtained attaching  a semi-infinite path $\{\mw_0,\mw_1,\mw_2\ldots\}$ to the root $\mv_0$ of a binary tree (i.e., letting $\mv_0=\mw_0$), and setting $a\equiv 1$ and $d\equiv 1$. Clearly, this is again a rooted tree and one can spot a good deal of symmetry in it. Still, $\mG$ it is not weakly spherically symmetric (e.g., the two nodes at distance $1$ from $\mv_0$ do not have same outdegree), and \emph{a fortiori} not spherically symmetric. By~\cite[Prop.~2.5]{BreKel11} the graph is not path commuting, and as a consequence neither the theory developed in~\cite{BreKel11} nor~\cite[Thm.~1]{KelLenWoj11} can be applied. However, $\mG$ has an equitable partition associated with cells $(\mV_{1,i},\mV_{2,i})_{i\in \mathbb N}$, where $\mV_{1,i}$ is the  $i$-th sphere of the binary tree and $\mV_{2,i}$ is a singleton consisting of the $i$-th node of the infinite path. Accordingly, we can apply Theorem~\ref{theo:equipart} and conclude that $e^{-t\Delta_2}$ commutes with the cellwise averaging operator for all $t\ge 0$.}

{
Observe however that if we consider the same graph $\mG$ choosing weights in a judicious way -- say, assigning a weight $1$ to each node and each edge of the binary tree and a weight $2^n$ to both the node $\mw_{n+1}$ and the edge $(\mw_n,\mw_{n+1})$ in the semi-infinite path, $n=1,2,\ldots$ -- then  one finds another equitable partition associated with cells $(\mV_{i})_{i\in \mathbb N}$, where $\mV_{i}$ is the  $i$-th (unweighted) sphere of the binary tree (in this case, $c_{ii}=0$, $c_{i\; i+1}=2$, and $c_{i+1\; i}=1$ for all $i=1,2,\ldots$), and in this case the weighted graph is weakly  spherically symmetric, too. It is still not spherically symmetric, though.}

{
(2) If $a\equiv 1$ and $d\equiv 1$, then the graph in~\cite[Fig.\ 4.(b)]{BreKel11} is spherically symmetric, hence has a non-trivial equitable partition, but it is not path commuting.
}

{
(3) If $a\equiv 1$ and $d\equiv 1$, then the comb lattice in $\mathbb Z^2$, cf.~\cite[\S~2.21]{Woe00}, is a tree with a non-trivial equitable partition -- say, with cells $(\mV_i)_{i\in \mathbb Z}$, where $\mV_i$ is the set of all nodes with ordinate $i$. However, no matter which node we consider as a root, there will be a sphere containing two nodes of indegree 2 and at least one node of indegree 1 -- hence this graph is not weakly spherically symmetric, hence not spherically symmetric, and by~\cite[Prop.~2.5]{BreKel11} not even path commuting.
}
\end{exa}

\section{Further discrete operators}\label{generalized}

In this section we comment on several different possible extensions of the theory introduced above.

\subsection{Generalized Laplacians}\label{sec:gen-lapl}

Fiedler, Colin de Verdi\`ere, and others have introduced and studied several variations on the classic discrete Laplacian. The following one, which goes back to~\cite{Col90}, is the most usual one: Given a simple connected graph with node set $\mV$, any $|\mV|\times |\mV|$-matrix whose off-diagonal $\mv$-$\mw$-entry is 
\begin{itemize}
\item $=0$ if and only there is no edge between $\mv$ and $\mw$ and
\item $<0$ otherwise
\end{itemize}
is called a \emph{generalized Laplacian}. Besides the discrete Laplacian, also $-\mathcal A$ (where $\mathcal A$ is the adjacency matrix) and Chung's normalized Laplacian (\cite[\S~1.4]{Chu97})
\[
{\mathcal P}:=D^{-\frac12} \Delta \,D^{-\frac12}
\]
(where $D$ is the diagonal matrix whose $i$-$i$-entry is the degree of $\mv_i$) are clearly generalized Laplacians. Unlike the adjacency matrix, Chung's $\mathcal P$ is also positive definite, since 
\[
{\mathcal P}=(D^{-\frac12}\mathcal I) (D^{-\frac12}\mathcal I)^T.
\] 
One of the most relevant motivations for working with the normalized Laplacian is that is a bounded operator, even if $\mG$ is not locally finite.

Let $\mG$ be additionally finite. The \emph{signless Laplacian}
\[
{\mathcal Q}:=D+A
\]
has been introduced in~\cite{DesRao94}. Its study has gained much momentum in the last decade, also owing to thorough investigations by Cvetkvovi\'c and others, cf.~\cite{CveRowSim07}. (Clearly, $-\mathcal Q$ is a generalized Laplacian, too).
One possible reason for this popularity is the richness of interesting graph theoretical properties of $\mathcal Q$ (e.g., $0$ is always an eigenvalue whose multiplicity is the number of bipartite components); another one is its nice variational structure. In particular, one easily sees that
\[
{\mathcal Q}:=\mathcal J \mathcal J^T,
\] 
where $\mathcal J$ is the incidence matrix of the undirected graph underlying $\mG$, i.e., $\mathcal J:=\mathcal I^++\mathcal I^-$. One may further generalize this class of Laplacians by taking $\sigma\in \ell^\infty_\nu(\mV)$ and letting
\[
(\mathcal J_\sigma)_{\mv\me}:=\mathcal I^+_{\mv\me} + \sigma(\mv)\mathcal I^-_{\mv\me},\qquad \mv\in \mV,\; \me \in \mE.
\]
We may then introduce
\[
{\mathcal Q}^\sigma:=\mathcal J_\sigma \mathcal J^T_\sigma,
\] 
which is still a positive definite matrix and also a generalized Laplacian (each of whose non-vanishing off-diagonal entries is one of the $\sigma_\mv$); along with its normalized version
\[
{\mathcal P}^\sigma:=(D^{-\frac12}\mathcal J_\sigma) (D^{-\frac12}\mathcal J_\sigma)^T.
\] (This generalization process could continue by allowing for more general matrices $\sigma$, but this would necessarily destroy locality of ${\mathcal Q}^\sigma$).
It is then natural to consider a \emph{signless $p$-Laplacian} by
\[
{\mathcal Q}_p f:=\mathcal J (|\mathcal J^T f|^{p-1}{\rm sign}(\mathcal J^T f)),
\] 
or, more generally, 
\[
\mathcal Q^\sigma_p f:=\mathcal J_\sigma (|\mathcal J_\sigma^T f|^{p-1}{\rm sign}(\mathcal J_\sigma^T f)).
\] 
Likewise, one can introduce the normalized operator
\[
\mathcal P^\sigma_p f:=D^{-\frac12}\mathcal J_\sigma (|\mathcal J_\sigma^T D^{-\frac12}f|^{p-1}{\rm sign}(\mathcal J_\sigma^T D^{-\frac12}f)).
\] 

We are aware of only a few previous investigations on the signless $p$-Laplacian in the literature, including~\cite{BiyHelLey09,ZhaZha11}, where it is proved that the sets of eigenvalues of $\mathcal Q_p$ and $\Delta_p$ agree, provided that $\mG$ is bipartite.

\begin{rem}
In the linear case $p=2$, the parabolic theory of the signless Laplacian on uniformly locally finite graphs is not overly interesting. This is due to the fact that $\mathcal Q+\Delta=2D$ (we are omitting the index $2$ of $\Delta^N_2,\mathcal Q_2$). All these three operators are bounded. Although $D$ and $\mathcal A$ (and therefore $D$ and $\Delta,\mathcal Q$) do not commute, so that in general
\[
e^{t\mathcal Q}\not=e^{-t\Delta}e^{2tD},\qquad t\in \mathbb R,
\]
Lie's product formula still holds and yields
\[
e^{t\mathcal Q}=\lim_{n\to\infty}\left(e^{-\frac{t}{n}\Delta}e^{2\frac{t}{n}D}\right)^n,\qquad t\in \mathbb R.
\]
Alternatively, we we can recover another formula for $(e^{t\mathcal Q})_{t\ge 0}$ based on the Dyson--Phillips formula. 
Since $D$ is a positive diagonal matrix, knowing qualitative information on $(e^{-t\Delta})_{t\in\mathbb R}$ promptly yields quite complete information on $(e^{t\mathcal Q})_{t\in\mathbb R}$, and vice versa. For example, since $(e^{-t\Delta})_{t\ge 0}$ is a positive $C_0$-semigroup, so is $(e^{t\mathcal Q})_{t\ge 0}$.

Things are different in the general nonlinear case, since it is not clear in which sense $\mathcal Q_p$ can be seen as a perturbation of $\Delta_p$ (unless $p$ is an even natural number, in which case $\Delta_p$ and $\mathcal Q_p$ can be compared using the binomial formula).
\end{rem}

 The following can be proved like in the case of $\Delta_p$. We now also allow for infinite graphs, i.e., we are in the same setting of the previous sections.

\begin{prop}
Let $p\in (1,\infty)$ and $\sigma\in \ell^\infty_\nu(\mV)$ and consider the functional $\Ffun_p^\sigma:\ell^2_d(\mV)\to [0,\infty]$ defined by
\[
\Ffun_p^\sigma:u\mapsto \frac{1}{p}\|\mathcal J^T_\sigma f(\me)\|^p_{\ell^p_a (\mE)} =\frac{1}{p}\sum_{\substack{\me\in \mE}} |f(\me_+)+\sigma(\me_-) f(\me_-)|^p.
\]

Then the following assertions hold.

\begin{enumerate}[(1)]
\item
For all $p\in [1,\infty]$, $z^{1,p,2}_{a,d}(\mV)$ is a Banach space with respect to the norm defined by
\[
\|f\|_{z^{1,p,2}_{a,d}}:=\|f\|_{\ell^2_d}+\|\mathcal J^T_\sigma f\|_{\ell^p_a}.
\]
For all $p\in [1,\infty]$, $z^{1,p,2}_{a,d}(\mV)$ is continuously and densely embedded into $\ell^2_d(\mV)$.
If moreover $p\in [1,\infty)$, then $z^{1,p,2}_{a,d}(\mV)$ is separable in $\ell^2_d(\mV)$. If $p\in (1,\infty)$, then $z^{1,p,2}_{a,d}(\mV)$  is uniformly convex and hence reflexive.
\item The functional $\Ffun_p^\sigma$ is convex. It is continuously Fréchet differentiable as a functional on $z^{1,p,2}_{a,d}(\mV)$, while it is lower semicontinuous as a functional on $\ell^2_d(\mV)$.
\item The subdifferential $\mathcal Q_p^\sigma:=\partial\Ffun_p^\sigma $ of $\Ffun_p^\sigma$ generates a $C_0$-semigroup of nonlinear contractions on $\ell^2_d(\mV)$.			
\item If $\sigma\equiv 1$, then the set of eigenvectors of $\mathcal Q_p^\sigma$ for the eigenvalue 0 form a subspace of $\ell^2_d(\mV)$ whose dimension is the number of bipartite components of $\mG$ with finite surface.
\end{enumerate}	
\end{prop}

Observe that unless $\sigma=\pm 1$, $\Ffun_p^\sigma$ and hence its subdifferential  $\mathcal Q_p^\sigma$ do depend on the orientation of $\mG$.

\begin{proof}
Because 
\begin{equation*}\label{eq-ep-2}
\Ffun_p^\sigma =\frac{1}{p}\|\cdot\|^p_{\ell^p_a}\circ \mathcal J^T_\sigma,
\end{equation*}
and because $\mathcal J_\sigma$ is a bounded linear operator from $z^{1,p,2}_{a,d}(\mV)$ to $\ell^p_a(\mE)$, all the proofs are analogous to those for the corresponding assertions about {$\Efun^N_p$} and $\Delta^N_p$. The only minor change is needed in the proof of (4), where the dimension formula for the null space of $\mathcal I^T=\mathcal J^T_{-1}$ has to be replaced by a dimension formula for the null space of $\mathcal J^T_1$, cf.~\cite{van76}. 
\end{proof}

\subsection{The $p(\me)$-Laplacian}

{
We have already mentioned in the introduction that the $p$-heat equation can be conveniently used for image processing. Let us shortly comment on this, also referring to the celebrated articles~\cite{PerMal90,CatLioMor92,CasCatCol93}, where geometric PDEs related to the $p$-heat equation were first considered with this purpose, and to the historical survey~\cite{CasMorSap93}. 
}

{
Assume the user is given a noisy b/w image (an analog one in the continuous case, a digital one in the present discrete setting; the color case may be also treated with minor technical modifications
). In  an attempt to restore the original image, a smart strategy consists in using the noisy image as the initial data of a $p$-heat equation, where the unknown $f$ describes how dark a point/pixel is. The basic idea is that a portion of the image can be safely smoothed if the gray tones are relatively homogeneous (low gradient), whereas a sharp increase in the value of $f$ (high gradient) suggests the existence of corners in the original image and should therefore be preserved. In other words, diffusion-driven smoothing should be stronger in regions of low gradient, and weaker in regions of high gradient. This can be obtained considering choosing $1\le p\ll 2$, while the opposite applies if $p\gg 2$, and then considering the equilibrium state towards which the system will (hopefully) converge. Essentially the same principle justify the usage of $p$-Laplacians for clustering tasks.}

{
Now, what if the user does in fact know the original image? In this case, one can argue that more accurate processing can be performed if the user is allowed, if needed, to modify the value $p$ in dependence of specific regions of the image. This leads to the introduction of the so-called $p(x)$-Laplacian.} In our discrete setting this amounts to consider a modified energy functional given by
\[
f\mapsto \sum_{\me \in \mE}\frac{a(\me)}{p(\me)}| f(\me_+)-f(\me_-)|^{p(\me)}
\]
under suitable assumptions on $p\in {\mathbb R}^\mE$. In this case it would not in general be its subdifferential that generates a semigroup, but rather the subdifferential of its convex, lower semicontinuous relaxation. We omit the details. This approach proves useful if, for instance, the user is aware of the fact that in a certain region of the image two different objects with similar colors are juxtaposed, so that smoothing that region is not safe after all.

\subsection{Discrete Schrödinger operators}

Another possible direction of generalization consists of considering an energy functional with additional terms defined by
\[
f\mapsto \frac{1}{p}\sum_{\me \in \mE}a(\me) |f(\me_+)-f(\me_-)|^p+\frac{1}{2}\sum_{\mv \in \mV}b(\mv) |f(\mv)|^2.
\]
for some $b$ asymptotically comparable with $d$. Due to their interpretation in (linear) potential theory, the terms in the second sum are sometimes referred to as \emph{killing terms}. The subdifferential of $\Efun_p$ is in fact the discrete analog of a Schrödinger operator with scalar potential $b$. We refer to~\cite{KelLen09} and references therein for a comprehensive theory in the linear case.

\subsection{Discrete operators with boundary conditions at infinity}\label{sec:disc-bdry-inf}

We have already remarked that whenever $\mG$ is finite, $\mathring{w}^{1,p,2}_{a,d}(\mV)$ always agrees with $w^{1,p,2}_{a,d}(\mV)$ (in fact, both agree with $\ell^2_d(\mV)$), but for infinite graphs things are less obvious, and a boundary of $\mG$ may arise. Several different notions of such a boundary exist: A characterization of the quotient space
$$w^{1,2,2}(\mG)/\mathring{w}^{1,2,2}(\mG).$$
 is a classic topic of potential theory and in special cases it leads to the introduction of the so-called Martin boundary (see e.g.~\cite[Chap.~IV]{Woe00} for its connection to the notion of space of ends of an infinite graph).
Some results in this direction in the general weighted case, including subtle discussions of parabolic and graph theoretical properties implying non-triviality of $w^{1,2,2}_{a,d}(\mG)/\mathring{w}^{1,2,2}_{a,d}(\mG)$, have been obtained in~\cite[\S~3]{SoaYam93} and~\cite{Pul10}, as well as --  for the case of $p=2$ -- in two series of papers by Jorgensen and Pearse, and by Keller and Lenz, conveniently surveyed in~\cite{JorPea11b} and~\cite{KelLen10}, respectively. 

One may introduce also in our case the quotient Banach space
$$b^p_{a,d}(\partial \mG):=w^{1,p,2}_{a,d}(\mG)/\mathring{w}^{1,p,2}_{a,d}(\mG).$$
If $\Tr$ denotes the canonical surjection of $w^{1,p,2}_{a,d}(\mG)$ onto $b^p_{a,d}(\partial\mG)$, one may consider the general functionals
$$\tilde{\Efun}_{p,q}(f):=\frac{1}{p}\|\mathcal I^T f\|^p_{\ell^p_a}+\frac{1}{q}\|\Tr f\|^q_{b^p_{a,d}(\partial \mV)},\qquad q>1.$$
Another possible approach relies upon the Gauss--Green-type formula for graphs developed in~\cite{JorPea09}, where the boundary sum may be exploited to treat discrete $p$-Laplacians with boundary conditions. 


\subsection{The porous medium equation}\label{sec:4}

We briefly discuss the interplay between porous medium equation and $p$-heat equation on graphs. For the sake of simplicity, we only consider the unweighted case, i.e., $\mu=a\equiv 1$ and $\nu=d\equiv 1$. The porous medium equation {in the continuum} can be studied in the context of the theory of subdifferentials, see e.g.~\cite[Examples~III.6.C and~IV.6.B]{Sho97}; 
but it is known that \emph{in the continuous, 1-dimensional case} if $\psi$ solves the porous medium equation
$$\dot{\psi}=(|\psi|^{\pi-1}\psi)_{xx},$$
then there is $\varphi$ that satisfies the $p$-heat equation 
$$\dot{\varphi}=(|\varphi_x|^{p-2}\varphi_x)_x$$
for $p:=\pi+1$ and such that $\psi$ is its pressure (i.e., $\varphi_x=\psi$). We can prove a similar result in a special class of oriented bipartite graphs.

While the proof in is based on the theory of exact differential forms (see e.g.~\cite[\S~3.4.3]{Vaz07}), and it is not clear whether it has a pendant in our context, this argument roughly suggests that the correct space to discuss the porous medium equation is $\ell^2(\mE)$, which we can look at as the node space of the line graph ${\mG_L}$ of $\mG$ and consider as a \emph{pressure space}. {(Recall that the line graph of an unweighted undirected graph $\mG$  is the graph $\mG_L=(\mV_L,\mE_L)$ with node set $\mV_L:=\mE$ and such that the edge $(\me,\mf)\in \mE_L$ if and only if $\me,\mf$ are adjacent edges in $\mG$).}

Observe that bipartition induces a natural orientation: if $\mV=\mV_1\cup\mV_2$, then assume all edges to have initial endpoint in $\mV_1$ and terminal endpoint in $\mV_2$. We recall that a bipartite graph  is called \emph{semiregular} if all nodes in $\mV_i$ have same degree .

\begin{prop}
Let $\mG$ be uniformly locally finite and $p\ge 2$. Let $T>0$ and $f\in L^2(0,T;\ell^2(\mV))$. If $\varphi\in H^1(0,T;\ell^2(\mV))$ satisfies the $p$-heat equation
\[
\dot{\varphi}(t,\mv)=-{\mathcal I}(|\mathcal I^T \varphi|^{p-2}\mathcal I^T \varphi)(t,\mv)+f(t),\qquad t>0,\; \mv \in \mV,
\]
then $\psi:=\mathcal I^T \varphi\in H^1(0,T;\ell^2(\mE))$ satisfies
\[
\dot{\psi}(t,\me)=-{\mathcal I}^T{\mathcal I}(|\psi|^{\pi-1}\psi)(t,\me)+\mathcal I^T f(t),\qquad t>0,\; \me \in \mE,
\]
for $p=\pi+1$. In particular, $\psi$ satisfies the porous-medium-type equation (with potential)
\begin{equation}\label{pme-semireg}
\dot{\psi}(t,\me)=\Delta_{{\mG_L}}(|\psi|^{\pi-1}\psi)(t,\me)-(r+s)(|\psi|^{\pi-1}\psi)(t,\me)+\mathcal I^T f(t),\qquad t>0,\; \me \in \mE,
\end{equation}
provided ${\mG}$ is additionally $(r,s)$-semiregular bipartite and is oriented accordingly.
\end{prop}

Observe that~\eqref{pme-semireg} is a \emph{forward} porous medium-type equation (with potential) on the line graph $\mG_L$ of $\mG$. We have denoted by $\Delta_{{\mG_L}}$ the (linear) discrete Laplacian on $\mG_L$. We stress that \emph{the latter result is not independent of orientation of $\mG$}.

\begin{proof}
It is a direct consequence of Lemma~\ref{lemma:car11} that $\dot{\psi}=\mathcal I^T \dot{\varphi}$ if $\mG$ is uniformly locally finite. 

Let now $\mG$ be bipartite and oriented accordingly. The second assertion follows recalling that
$${\mathcal I}^T{\mathcal I}=(r+s){\rm Id}-\Delta_{{\mG_L}}$$
whenever $\mG$ is $(r,s)$-semiregular, see e.g.~\cite[Proof of Thm.~3.9]{Moh91}.
\end{proof}

 Consequently, by Corollary~\ref{cor:bije} the following holds, since an infinite graph has infinite surface if $\nu\equiv 1$.

\begin{cor}
Let $T>0$. If in particular $\mG$ is a forest each of whose connected components is an infinite, uniformly locally finite tree, then for all $g\in L^2(0,T;\ell^2(\mE))$ and all $\psi_0\in \ell^2(\mE)$ there exists a unique solution $\psi\in H^1(0,T;\ell^2(\mE))$ to 
\begin{equation}
\tag{PME}
\left\{
\begin{array}{rcll}
\dot{\psi}(t,\mv)&=&{\mathcal I}^T{\mathcal I}(|\psi|^{\pi-1}\psi)(t,\mv)+g(t),\qquad &t\ge 0,\; \mv\in \mV,\\
\psi(0,\mv)&=&\psi_0(\mv),& \mv\in \mV.
\end{array}
\right.\end{equation}
\end{cor}

\section{{Appendix: A reminder of the theory of subdifferentials}}

{
The community of nonlinear analysts and that of functional analysts on graphs seem to have grown more and more disjoint in recent years. We have thus chosen to recollect some basic results about nonlinear semigroups generated by subdifferentials for the reader less familiar with this theory.}

\begin{assum}\label{assum:append}
Throughout this section we consider the following functional setting.
\begin{itemize}
\item $V$ is a reflexive Banach space.
\item $H$ is a Hilbert space.
\item $V$ is densely and continuously embedded in $H$.
\item $\Efun : V \to \mathbb [0,\infty)$ is a convex, Fréchet differentiable functional (with derivative denoted by $\Efun'$) such that $\Efun(0)=0$.
\end{itemize}
\end{assum}
We can and will extend $\Efun$ to the whole $H$ by $+\infty$: With an abuse of notation we denote this extension again by $\Efun$. Then the \emph{subdifferential} of $\Efun:H\to [0,+\infty]$ at $f\in H$ is defined (cf.\ e.g.~\cite[Def.\ at p.\ 81]{Sho97}) as the  set
\[
\partial \Efun(f):=
\left\{
\begin{array}{ll}
\{g\in H:(g|\varphi-f)_H \le \Efun (\varphi)-\Efun(f) \; \forall \varphi\in H\},\qquad &\hbox{if }f\in V,\\
\emptyset,\qquad &\hbox{if }f\in H\setminus V,\\
\end{array}
\right.
\]
However, under the above assumptions the subdifferential of $\Efun$ at each $f\in V$  is by~\cite[Prop.~II.7.6]{Sho97} either empty or a singleton. Thus, we can regard $\partial\Efun$ as a (single-valued) operator from
\[
D(\partial \Efun):=\{f\in V:\partial\Efun(f)\not=\emptyset \}
\]
to $H$, and we denote with a slight abuse of notation
\[
\partial \Efun (f)\equiv\{\partial\Efun(f)\},\qquad f\in D(\partial \Efun).
\]

{
Determining a subdifferential is in general a tedious task.  Due to our assumption of Fréchet differentiability of $\Efun$, however, a subdifferential can be described more easily by means of the following result. While it seems to be known, the only precise reference we are aware of is~\cite[Lemma~2.8.9]{Nit10}.
\begin{lemma}\label{Lemmanittka}
The subdifferential of $\Efun$ can be equivalently described by
\begin{equation}\label{nittefun}
\begin{array}{rcl}
D(\partial\Efun)&=&\{f\in V:\exists g\in H \hbox{ s.t.\ }\Efun'(f)h=(g|h)_H\;\; \forall h\in V\},\\
\partial\Efun (f)&=&g.
\end{array}
\end{equation}
\end{lemma}
Unlike in the linear case, in the world of nonlinear evolution equations several discording techniques for finding solutions exist.  In particular, subdifferentials of proper, convex, lower semicontinuous functionals are (nonlinear) $m$-accretive operators, cf.~\cite[Prop.~IV.2.2]{Sho97}. Therefore  the celebrated Crandall--Liggett Theorem, the nonlinear pendant  of the theorem of Lumer--Phillips, can be applied to find solutions of abstract Cauchy problems associated with subdifferentials in terms of a semigroup of nonlinear operators. Further significant information abouth solutions can be obtained applying more directly the properties of subdifferentials and using Hilbert space methods, in particular when considering  inhomogeneous abstract Cauchy problems. 
}

{In the following we summarize several different results obtained with different methods, and in particular three celebrated results by Brezis, Crandall--Liggett, and Kato, cf.~\cite[Thm.~IV.4.1, Thm.~IV.4.3, and Thm.~IV.8.2]{Sho97} or~\cite[Théo.~3.1, Théo.~3.2, Théo.~3.3, Théo.~3.6]{Bre73}.
\begin{theo}\label{thm:brezkato}
Let $T>0$. Then for all $f\in L^2(0,T;H)$ and all $f_0\in H$ there exists a unique $\varphi\in C([0,T];H)$ which is differentiable for a.e.\ $t\in [0,T]$ and such that
\begin{equation}
\left\{
\begin{array}{rcll}
\varphi(t)&\in& D(\partial \Efun), &\hbox{for a.e. }t\in [0,T],\\
\dot{\varphi}(t)+\partial \Efun \varphi (t)&=& f(t),\qquad &\hbox{for a.e. }t\in [0,T],\\
\varphi(0)&=&f_0.
\end{array}
\right.\end{equation}
Furthermore, $\Efun \circ u \in L^1(0,T)$, and moreover 
\begin{itemize}
\item if $f_0\in V$, then $\Efun \circ u \in L^\infty(0,T)$ and $u\in H^1 (0,T;H)$;
\item if $f_0\in D(\partial \Efun)$, then $u\in W^{1,\infty}(0,T;H)$ and $u$ is right differentiable for all $t\ge 0$,
\item if $f\equiv 0$, then the mapping defined by
\[
e^{-t\Efun }:H\ni f_0\mapsto u(t)\in H,\qquad t\ge 0,
\]
forms a nonlinear contractive $C_0$-semigroup, i.e., a strongly continuous family of (in general, nonlinear) contractions on $H$ that satisfy the semigroup law. These operators can also be obtained by
\begin{equation}\label{eq:exp-formula}
e^{-t\partial \Efun} f_0=\lim_{n\to \infty} J_\frac{\lambda}{n}^n(\Efun) f_0,\qquad t\ge 0,
\end{equation}
where the operators $J_\cdot(\Efun)$ are defined by
\begin{equation}\label{resolv-nonl}
J_\lambda(\Efun):=({\rm Id}+\lambda \partial \Efun)^{-1},\qquad \lambda \in (0,\infty).
\end{equation}
\end{itemize}
\end{theo}
An alternative approach to investigate well-posedness of nonlinear evolution equations goes back to~\cite{Lio69} and has been substantially enhanced by Chill and coauthors in recent years, cf.~\cite{ChiFas10} for a comprehensive exposition. The following collects~\cite[Thm.~6.1 and~\S~6.4]{ChiFas10}.
\begin{theo}\label{theo:galerkinsch}
Additionally to our standing assumptions, let 
\begin{itemize}
\item $V$ be separable,
\item $\Efun$ be \emph{coercive}, i.e., the sublevel sets
\begin{equation*}
\label{eq:levelsets}
\{f\in V: \Efun(f)\le \alpha \}
\end{equation*}
be bounded with respect to the norm of $V$ for all $\alpha\in \mathbb R$, and 
\item the Fréchet derivative $\Efun'$ map bounded sets of $V$ into bounded sets of $V'$. 
\end{itemize}
Then for all $f\in L^2(0,T;H)$ and all $f_0\in V$ there exists a unique $\varphi \in L^\infty(0,T;V)\cap  H^1 (0,T;H)$ such that
\begin{equation}\label{(ACP)}
\left\{
\begin{array}{rcll}
\varphi(t)&\in& D(\partial \Efun), &\hbox{for a.e. }t\in [0,T],\\
\dot{\varphi}(t)+\partial \Efun \varphi (t)&=& f(t),\qquad &\hbox{for a.e. }t\in [0,T],\\
\varphi(0)&=&f_0.
\end{array}
\right.\end{equation}
Furthermore, the energy inequality
\begin{equation}
\label{eq:energineq}
\int_0^t \|\dot{\varphi}(s)\|^2_H ds+ \Efun( \varphi(t)) \le \Efun (f_0) +\int_0^t (f(s)|\dot{\varphi}(s))_H ds,\qquad t\in [0,T],
\end{equation}
is satisfied. If in particular $V=H$ is finite dimensional and $f\equiv 0$, then the solution $\varphi$ satisfies the further energy inequality
\begin{equation}
\label{eq:degiorgi}
\frac{d}{dt} \Efun(\varphi(t)) \le -\frac12 \|\dot{\varphi}(t)\|^2_H-\frac12 \|\partial \Efun {\varphi}(t)\|^2_H,\qquad t\ge 0.
\end{equation}
\end{theo}
The main idea of the proof is to consider the weak formulation 
\[
(\dot{\varphi}(t)|h)_H+\partial \Efun \varphi (t)h= \left(f(t)|h \right)_H,\qquad \hbox{for a.e. }t\in [0,T]\hbox{ and all }h\in V,
\]
or rather
\[
(\dot{\varphi}(t)|h)_H+\Efun'( \varphi (t))h= \left(f(t)|h \right)_H,\qquad \hbox{for a.e. }t\in [0,T]\hbox{ and all }h\in V,
\]
of the differential equation, and then to discretize it by applying the Galerkin scheme. Let us briefly sketch the main steps of this argument, since they will play a r\^ole in Section~\ref{sec:3}:
\begin{enumerate}[(1)]
\item since $V$ is separable, one can take
\begin{itemize}
\item  a total sequence $(e_n)_{n\in \mathbb N}$ and hence the sequence of finite dimensional spaces $V_n:={\rm span}\{e_m:m\le n\}$ (with the norm induced by $H$) such that $\bigcup_{n\in \mathbb N}V_n$ is dense in $V$ and
\item a sequence $(f_{0n})_{n\in \mathbb N}$ such that $f_{0n}\in V_n$ for all $n\in \mathbb N$ and $\lim_{n\to \infty} f_{0n}=f_0$ in $V$;
\end{itemize} 
\item for all $n\in \mathbb N$, consider $\Efun_{|V_n}$, take its subdifferential as in~\eqref{nittefun}, but w.r.t.\ test functions in $V_n$; and use Carathéodory's Theorem to solve
\begin{equation*}
\left\{
\begin{array}{rcll}
\dot{\varphi}_n(t)+\partial \Efun_{|V_n} \varphi_n (t)&=& P_n f(t),\qquad &\hbox{for a.e. }t\in [0,T],\\
\varphi_n (0)&=&f_{0n},
\end{array}
\right.\end{equation*}
where we denote by $P_n$ the orthogonal projection of $H$ onto $V_n$;
\item  for all $n\in \mathbb N$, show that all the solutions $\phi_n$ admit uniform a priori bounds, which in turn show that the sequence $(\phi_n)_{n\in \mathbb N}$ is bounded  in $L^\infty(0,T;V)\cap  H^1 (0,T;H)$ and also that $(\Efun (\phi_n))_{n\in \mathbb N}$ is bounded in $L^\infty(0,T;V')$;
\item extract a converging subsequence and show that its limit is a solution of the abstract Cauchy problem~\eqref{(ACP)}  with the claimed regularity properties and satisfying the energy inequality~\eqref{eq:energineq};
\item use accretivity of $\partial \Efun$ to prove that there cannot be further solutions.
\end{enumerate}
The latter energy inequality~\eqref{eq:degiorgi}, which in~\cite[\S~6.4]{ChiFas10} is reported to be due to De Giorgi, shows that either the solution reaches in finite time a ground state, or its energy decreases indefinitely.}

\begin{rem}
Let us recall the connection between the linear theory of quadratic forms and the nonlinear theory we are summarizing in this section: If under our standing assumptions ${\mathfrak h}:V\times V\to \mathbb R$ is a bounded, coercive, symmetric, bilinear form (that is, a \emph{quadratic form}), then 
\[
\Efun:V\ni f\mapsto \frac12 {\mathfrak h}(f,f)\in [0,\infty)
\]
defines a convex, coercive, Fréchet differentiable functional and all sublevel sets are bounded in $V$, hence it satisfies the assumptions of all results in this section. Moreover, one sees directly that the Fréchet derivative of $\Efun$ is given by
\[
\Efun'(f)g={\mathfrak h}(f,g),\qquad f,g\in V.
\]
Therefore, by definition the subdifferential of $\Efun$ is precisely the linear operator associated with $\mathfrak h$.
\end{rem} 

{Finally, making use of semigroup theory it is possible to characterize closed convex sets of $H$ that are left invariant over time, by a nonlinear generalization of the Beurling--Deny criteria 
due to Barthélemy. The following combines~\cite[Théo~1.1 and Cor.~2.2]{Bar96},~\cite[Prop.~4.5]{Bre73}, and~\cite[Cor.~3.7]{CipGri03}.
\begin{lemma}\label{Bar96}
Let $C$ be a closed convex subset of $H$ and denote by $P_C$ the orthogonal projection of $H$ onto $C$. Then the following assertions are equivalent.
\begin{enumerate}[(i)]
\item $C$ is left invariant  under $J_\lambda(\Efun)$ for all $\lambda>0$.
\item $C$ is left invariant  under $e^{-t\partial \Efun}$ for all $t\ge 0$.
\item $\Efun (P_C f_0)\le \Efun (f_0)$ for all $f_0\in V$.
\end{enumerate}
In particular:
\begin{itemize}
\item The semigroup $(e^{-t\partial \Efun})_{t\ge 0}$ is order preserving, i.e.,
\begin{equation}
\label{eq:order-pres}
f_0\le g_0 \quad \Rightarrow\quad e^{-t{\mathcal L}}f_0\le e^{-t{\mathcal L}}g_0\qquad \forall t\ge 0,\;\forall f_0,g_0\in H,
\end{equation}
if and only if
\begin{equation*}
\label{cipgrieq1}
\Efun (f_0 \wedge g_0) + \Efun (f_0 \vee g_0) \leq \Efun (f_0) + \Efun (g_0)\quad\hbox{for all } f_0,g_0 \in V.
 \end{equation*}
\item If $H=L^2(X;\lambda)$ for a $\sigma$-finite measure space $(X,\lambda)$, then $(e^{-t\partial \Efun})_{t\ge 0}$ is contractive with respect to the norm of $L^\infty(X,\lambda)$, i.e., 
$$\|e^{-t\partial \Efun}f_0-e^{-t\partial \Efun}g_0\|_\infty \le \|f_0-g_0\|_\infty \qquad \forall t\ge 0,\;\forall f_0,g_0\in H,$$
if and only if
\begin{eqnarray*}\label{cipgrieq2}
&&\Efun \left(\frac{g_0+(f_0-g_0+1)_+}{2} + \frac{g_0-(f_0-g_0-1)_-}{2}\right)\\
&&\qquad+\Efun \left(\frac{f_0-(f_0-g_0+1)_+}{2} + \frac{f_0+(f_0-g_0-1)_-}{2}\right) \leq \Efun (f_0) + \Efun (g_0)\quad\hbox{for all } f_0,g_0 \in V.
 \end{eqnarray*}
\end{itemize} 
\end{lemma}
}

In particular, we can use Proposition~\ref{Bar96} to prove the following.

\begin{cor}\label{commut_lemma}
Let $\tilde{V}, \tilde{H},\tilde{\Efun}$ be a further Banach space, a further Hilbert space, and a further functional, respectively, which satisfy the Assumptions~\ref{assum:append}

Let $\Sigma$ be a bounded linear operator from $H$ to $\tilde{H}$. Then $\Sigma$ intertwines with the $C_0$-semigroup on $H$ and $\tilde{H}$ generated by $-\partial \Efun$ and $-\partial \tilde{\Efun}$, respectively, i.e.,
\[
e^{-\partial \Efun}\Sigma=\Sigma e^{-t\partial \tilde{\Efun}}\qquad \hbox{for all }t\ge 0
\]
if and only if
\[
\partial \Efun(Lf_0 + \Sigma^* R g_0)+\partial \tilde{\Efun}(\Sigma L f_0 + g_0-Rg_0)\le \partial \Efun(f_0)+\partial \tilde{\Efun}(g_0)\qquad \hbox{for all }f_0\in H,\; g_0\in \tilde{H},
\]
where
\[
L:=({\rm Id}_H+\Sigma ^*\Sigma )^{-1}\qquad \hbox{and}\qquad R:=({\rm Id}_{\tilde{H}}+\Sigma\Sigma ^* )^{-1}.
\]
\end{cor}

\begin{proof}
One checks directly that $\Sigma$ intertwines with the semigroups if and only if the graph of $\Sigma$, i.e., the closed subspace
$${\rm Graph}(\Sigma):=\left\{\begin{pmatrix}f_0\\ \Sigma f_0\end{pmatrix}\in H\times \tilde{H}\right\}$$
is invariant under the matrix semigroup
$$e^{-t\partial {\bf E}}:=\begin{pmatrix}
e^{-t\partial \Efun} & 0\\ 0 & e^{-t\partial \tilde{\Efun}}
\end{pmatrix},\qquad t\ge 0,$$
where ${\bf E}:=\Efun\oplus \tilde{\Efun}$.
A classical formula due to von Neumann yields that the orthogonal projection of $H\times \tilde{H}$ onto ${\rm Graph}(\Sigma)$ is given by
$$P_{{\rm Graph}(\Sigma)}=\begin{pmatrix}
L & \Sigma ^*R\\
\Sigma L & {\rm Id}_{\tilde{H}}-R
 \end{pmatrix},$$
cf.~\cite[Thm.~23]{Neu97} (it is also proved therein that $L,R$ exist as bounded linear operators). Now, the assertion follows from Proposition~\ref{Bar96}.
\end{proof}

\bibliographystyle{abbrv}
\bibliography{../../referenzen/literatur}
\end{document}